\documentclass[12pt]{article} 
\usepackage[top=1in,bottom=1in,right=.8in,left=.8in]{geometry}
\usepackage{amssymb,amsmath,amsthm}
\usepackage{tikz}
\usepackage{arydshln}
\usepackage{verbatim}
\usepackage{titlesec}
\usepackage{bbm}
\usepackage{bm}
\usepackage{blkarray}
\usepackage{subfigure}
\usepackage{stmaryrd}
\usepackage{float}
\usepackage[linesnumbered,ruled]{algorithm2e}
\usepackage{titling}
\usepackage{hyperref}

\usepackage{caption}
\usepackage{cleveref}
\usepackage{xcolor}
\usepackage{color}
\crefname{equation}{}{}
\crefrangeformat{equation}{(#3#1#4) --~(#5#2#6)}
\crefmultiformat{equation}{(#2#1#3)}{, (#2#1#3)}{, (#2#1#3)}{, (#2#1#3)}
\crefname{lem}{Lemma}{Lemmas}
\crefname{section}{Section}{Sections}
\crefname{subsubsubsection}{Section}{Sections}
\crefname{rem}{Remark}{Remarks}
\crefname{figure}{Figure}{Figures}
\crefname{table}{Table}{Tables}
\Crefname{lem}{Lemma}{Lemmas}
\crefname{thm}{Theorem}{Theorems}
\Crefname{thm}{Theorem}{Theorems}
\Crefname{question}{Question}{Questions}

\newtheorem{thm}{Theorem}[section]

\newtheorem{remark}{Remark}[section]
\newtheorem{rem}[thm]{Remark}
\newtheorem{lem}{Lemma}[section]

\newtheorem{proposition}[thm]{Proposition}

\newtheorem{corollary}[thm]{Corollary}

\newtheorem{question}[thm]{Question}
\theoremstyle{definition}

\theoremstyle{definition}
\newtheorem{defn}{Definition}[section]
\renewenvironment{proof}{\noindent {\bf Proof.}}{\qed}
\newcommand\p{{\bf p}}
\newcommand\q{{\bf q}}

\newcommand{\bp}{{\mathbf p}}
\newcommand{\bq}{{\mathbf q}}

\title{Globally Rigid Convex Braced Polygons}

\author{Robert Connelly\thanks{Department of Mathematics, Cornell University, Ithaca, USA.}
\and Bill Jackson\thanks{School of Mathematical Sciences Queen Mary University of London, England, UK} \and Shin-ichi Tanigawa\thanks{Department of Mathematical Informatics Graduate School of Information Science and Technology
University of Tokyo, Japan} \and  Zhen Zhang\thanks{Yau Mathematical Sciences Center, Tsinghua University, Beijing, China.}}
\date{September 2024}

\begin{document}

\maketitle

\begin{abstract}
Determining the global rigidity of a bar framework is known to be a difficult problem.  Here we propose a class of frameworks in the plane, convex braced polygons, that may be globally rigid and are analogous to convex polyhedra in $3$-space, which are globally rigid in the class of convex polyhedra by a celebrated theorem of Cauchy from 1813.  We show that  all the strictly convex plane realisations of the underlying graph of a convex braced polygon are infinitesimally rigid  if and only if they are all globally rigid in the class of convex braced polygons.  Furthermore, we provide techniques that, for many such graphs, prove that all their strictly convex plane realisations  are globally rigid, and other techniques to find instances when there is a strictly convex plane realisation that is not globally rigid. 
\end{abstract}

\section{Introduction}

\subsection{History}\label{Sect:History}

One of the most fundamental breakthroughs in understanding the rigidity of structures is Cauchy's 1813 result that a convex triangulated surface is ``rigid". (See \cite{Handbook}, vol. A 223–271.)  This is despite a mistake in his proof of what is now called his ``arm lemma", and also a lack of interest in the subject of rigidity for many years.  Nevertheless, the ideas in his proof are correct, and with the passage of time, interest in the subject of rigidity has increased.  (See ``Rigidity through a projective lens" in \cite{projective}.)  The statement of Cauchy's theorem is clear.  There is only one way to assemble a strictly convex triangulated polyhedron, up to congruent copies, that is convex and where vertex adjacencies are preserved. As a consequence, there is no continuous motion of the polyhedron, where each (triangular) face remains congruent during the motion.  The polyhedron is (locally) rigid.

In 1916 Max Dehn \cite{Dehn} proved that the same strictly convex triangulated polyhedron is {\it infinitesimally rigid} in $\mathbb{R}^3$.  This means that, when there are vectors $\p'=(p'_1, \dots, p'_n)$ attached to the corresponding vertices of the configuration $(G,\p)$ that act as the first derivative of a finite motion, preserving edge (bar) lengths, then they form the derivative of a congruence.  This is equivalent to saying that the rank of a matrix, called the \emph{rigidity matrix} $R(\p)$, is maximal, namely $3n-6$, where $n$ is the number of vertices of the polyhedron.  This also implies that there is no continuous motion, a flex, of the framework that preserves the length of each edge.  In the case of a triangulated polyhedron,  each (triangular) face remains congruent during the motion.  A priori, Dehn's theorem looks to be somewhat weaker than Cauchy's theorem, and, indeed, in 1975 Herman Gluck \cite{Gluck} showed that there is a relatively easy way to prove Dehn's infinitesimal theorem using the combinatorial idea of Cauchy's theorem.  It turns out, anyway, that it is easy to show that Cauchy's theorem directly implies Dehn's theorem using Proposition \ref{push-pull} below, which is related to some averaging techniques going back to \cite{averaging}, \cite{Saliola-Whiteley}, and feels like something Minkowski would have done.  

Meanwhile, the question of determining the ``rigidity" of a bar framework, even in the plane, was considered as a computational question.  
 In this not-necessarily convex setting, it is often assumed that the coordinates of the configuration $\p$ are {\it generic}, that is, they do not satisfy any non-zero polynomial equation with integer coordinates. In this case, the (local) rigidity of the framework $(G,\p)$ only depends on the graph $G$ and not on the particular configuration in  $\mathbb{R}^d$.  (Geiringer, Laman)\cite{Geiringer, Laman} characterised graphs which are generically rigid  in $\mathbb{R}^2$, and Lov\'asz and Yemini \cite{Lovasz}, used matriod methods to show there is a combinatorial polynomial-time algorithm to determine the generic rigidity of a graph. Indeed, Bruce Hendrickson and Jacobs \cite{Hendrickson}  have an elementary algorithm, the pebble game, that is easily implemented to compute the generic rigidity of a graph $G$ in the plane.
 
 In $\mathbb{R}^3$ and higher, nobody has found a polynomial-time deterministic algorithm to determine generic rigidity of the graph $G$.  On the other hand, one can simply compute the rank of the rigidity matrix $R(\p)$ for $(G,\p)$ for a randomly chosen configuration $\p$ in $\mathbb{R}^d$, and with high probability that will determine the generic rigidity of the graph $G$.

With Cauchy's result about strict convexity, there is no need to assume that the configuration is generic.  The convexity assumption is enough. On the other hand, Gluck noticed that since the space of strictly convex realizations is an open set, it automatically implies that the graph of a triangulated sphere is generically rigid in $\mathbb{R}^3$.  (See also the result of Fogelsanger \cite{Fogelsanger} about the generic rigidity of other triangulated manifolds.)

Another, closely related notion of local rigidity, and infinitesimal rigidity, is the {\it global rigidity} of a bar framework $(G,\p)$ in  Euclidean space $\mathbb{R}^d$.  This means that $(G,\p)$ is the only configuration in $\mathbb{R}^d$ with the corresponding bars the same length, up to congruent copies.  Global rigidity can be much harder to test, but if we are only interested in determining the global rigidity of a framework $(G,\p)$ at generic configurations, the problem becomes more tractable.

Still there is another problem, even if the configuration $\p$ is known exactly.  A test for generic global rigidity (GGR) is that there is an equilibrium stress $\omega$ for a given framework $(G,\p)$ with a corresponding stress matrix $\Omega$ that has maximal rank.  The top two frameworks of Figure \ref{fig:Jackson} both have a one-dimensional stress with a stress matrix $\Omega$ that has maximal rank. In this case the rank is $n-d-1=8-2-1=5$, where $n=8$ is the number of vertices, and $d=2$ is the dimension of the ambient space.  The frameworks are both infinitesimally rigid in the plane, and the vertices form strictly convex polygons.  The two configurations have corresponding edge lengths the same and are not congruent, even though they testify to the generic rigidity of the corresponding graph. When one knows that the graph $G$ is generically globally rigid, in practice, it can be extremely tedious, and at least difficult, to determine the global rigidity of a given framework $(G,\p)$, even in the plane.  (In \cite{Global} it is shown that this is an NP-complete problem to determine.)  It is possible, with a given graph $G$, to find frameworks that are globally rigid in all higher dimensions, starting in a given lower dimension, universal rigidity, as in \cite{OT,Oba}, but this may not deal with one's favorite configuration.   Knowing that a graph is generically globally rigid in a  given dimension does not tell us anything about the global rigidity of a particular realisation.

\subsection{Our results}\label{Sect:results}

If one is given a particular framework, or a class of frameworks, as described below, how can one determine global rigidity?  We propose, here, some tools to do that. 

In Section \ref{Sect:definitions}, we consider the class of strictly ``convex braced polygons".
 We show that  all the strictly convex plane realisations of the underlying graph of a convex braced polygon  are infinitesimally rigid  if and only if they are all globally rigid in the class of convex braced polygons (Theorem \ref{convex}). 
This makes it easier to test if a particular configuration in that class has a strictly convex configuration that is not globally rigid, given that one can know when it is infinitesimally rigid.  For general graphs, it might be more helpful to use something like the ``pure condition" in \cite{pure-condition} to find the critical configurations, where the rigidity matrix $R(\p)$ drops rank, for instance. See also \cite{projective} for a history of some of the classical methods for computing the (infinitesimal) rigidity of bar frameworks.

In Section \ref{Sect:Braced}, we show that, if a graph of a convex braced polygon  is minimally 3-connected with respect to its set of braces, then the property that all its strictly convex plane realisations are globally rigid  (the property studied in Section~\ref{Sect:definitions}) is equivalent to the property
that every strictly convex plane realization has a proper stress, i.e., a stress which is positive on the boundary edges  and negative on the braces (Theorem~\ref{thm:minimally}). 
The latter property implies global rigid by a theorem of Connelly\cite[Theorem 5.14.4]{book}.
Theorem~\ref{thm:minimally} is closely related to an earlier result of Geleji and Jord{\'a}n~\cite{robust} on convex braced polygons whose underlying graph is a 3-connected circuit in the generic 2-dimensional rigidity matroid.
 Geleji and Jord{\'a}n~\cite{robust} characterized when such a graph has  the property that all its strictly convex realizations have 
 a proper stress.
We shall give a new characterization of this property in Theorem \ref{m3p}. 
Geleji and Jord{\'a}n~\cite{robust} focused on a special combinatorial property of the graphs of convex braced  polygons, called the ``unique interval property" (Definition \ref{interval}). Our proof of Theorem \ref{m3p} makes no mention of that property, and uses minimal 3-connectivity instead.
We will derive a combinatorial relation between minimal 3-connectivity and the unit interval property in the appendix.


In Section \ref{Sect: Examples}, we will give some examples that cannot be resolved by the results in the previous sections. We describe some techniques which can be used to verify their global rigidity as well as an averaging technique which can verify non-global rigidity.  One very useful method is to use superposition of configurations that can work very efficiently with both the stress matrix $\Omega$ and the rigidity matrix $R(\p)$. These techniques can be used to determine whether any graph of a convex braced polygon with at most 7 vertices is globally rigid at all strictly convex realisations in the plane, see  Figure \ref{fig:super}.  We also provide a table of the graphs of convex braced polygons on at most 7 vertices whose strictly convex realisations are globally rigid when restricted to the class of all strictly convex realisations, see Figure \ref{fig:convexly}.
In some cases, one can do calculations simply by seeing how configurations cross thresholds without doing computer calculations.  See Figures \ref{fig:visual} and \ref{fig:plates} for an example.
(This reminds one of some of the ``catastrophe" techniques in \cite{Catastrophe}, although Catastrophe Theory, itself, does not seem helpful.)   A graph  on eight vertices for which we cannot determine whether all strictly convex realisations are globally rigid is also provided (Figure \ref{fig:Grunbaum}), together with other unsolved natural questions.

\section{Convex Braced Polygons}\label{Sect:definitions}

We propose a class of planar frameworks, motivated from Cauchy's Theorem in dimension three, where it may be easier to determine their infinitesimal rigidity, and even global rigidity.

 A $d$-dimensional \textit{bar framework} is a pair $(G,\p)$ , where $G$ is a finite graph whose vertices correspond to a {\it configuration} $\p=(p_1,\dots, p_n)$ of points $p_i$, $i=1\dots, n$ in $d$-dimensional  Euclidean space.  The edges of $G$ correspond to fixed length bars, but not necessarily of the same length, connecting the vertices of $(G,\p)$.
 We will also refer to $(G,\p)$ as a {\em realisation} of $G$ in $\mathbb{R}^d$.
 We will say that $(G,\p)$ and $(G,\q)$ {\em have corresponding bars the same length} to mean that, for each edge $ij$ of $G$, $\|p_i-p_j\|= \|q_i-q_j\|$. Two $d$-dimensional frameworks $(G,\p)$ and $(G,\q)$ are {\it congruent} if $\|p_i-p_j\|=\|q_i-q_j\|$, for all $1\le i <j\le n$.

\begin{defn}
For a bar framework $(G,\p)$ in $\mathbb{R}^d$, and $\p'=(p'_1,\dots, p'_n)$, regarded as a configuration of vectors $p'_i$ in $\mathbb{R}^d$, we say $\p'$ is an {\em infinitesimal flex} of $(G,\p)$ if, for each edge $ij$ of $G$, we have
\[
(\p_i-\p_j)\cdot (\p_i'-\p'_j)=0.
\]
We say that such a $\p'$ is {\it trivial} if it is the derivative of a smooth family of motions of congruences of $\mathbb{R}^d$ restricted to $\p$.
A bar framework $(G,\p)$ in $\mathbb{R}^d$ is {\it infinitesimally rigid} if every infinitesimal flex $\p'$ of $(G,\p)$ is trivial.
\end{defn}

\begin{defn}
A bar framework $(G,\p)$ in $\mathbb{R}^d$ is {\it globally rigid} in $\mathbb{R}^d$ if every other bar framework $(G,\q)$ in $\mathbb{R}^d$,  whose corresponding edge lengths are the same,  is congruent to $(G,\p)$. 
\end{defn}

\begin{defn}
A bar framework $(G,\p)$ in $\mathbb{R}^d$ is {\it universally rigid}  if every other bar framework $(G,\q)$ in $\mathbb{R}^D$, for all $ D \ge d$,  whose corresponding edge lengths are the same,  is congruent to $(G,\p)$. 
\end{defn}

A convex polygon in the plane is {\it strictly convex} if each vertex  is in a line that intersects the polygon only at this vertex.
Suppose that a (strictly) convex polygon in the plane has edges with fixed length and some additional edges, also of fixed length, connecting some of the polygon vertices. We call this a {\it (strictly) convex braced polygon}, and we regard it as a bar framework in the plane.  
For a convex braced polygon we assume that the vertices $(p_1,\dots, p_n)$ form a convex polygon with each edge $i,i+1$, modulo $n$, corresponding to a boundary edge of the polygon.
The underlying graph of a convex braced polygon is called a {\em braced polygon graph}.
It is a Hamiltonian graph with a distinguished Hamilton cycle.

Motivated by Cauchy's rigidity theorem, we introduce the following rigidity notion for strictly convex braced polygons. 

\begin{defn}
A braced strictly convex polygon $(G,\p)$ in $\mathbb{R}^2$ is {\it convexly rigid} if every other strictly convex braced polygon $(G,\q)$,  whose corresponding edge lengths are the same,  is congruent to $(G,\p)$. In addition, we say that a braced polygon graph is {\em convexly rigid} if all its strictly convex configurations are convexly rigid.   
\end{defn}

We are concerned here with the problem of deciding when a braced polygon graph has the property that all its strictly convex realisations are rigid in four senses, infinitesimally rigid, convexly rigid, globally rigid, and universally rigid.  
The first result in this paper is the following theorem.

\begin{thm}\label{convex}
 For a given braced polygon graph $G$, all strictly convex braced polygon frameworks $(G,\p)$ are infinitesimally rigid if and only if they are all convexly rigid.
\end{thm}

One direction of this theorem follows from the following proposition, which is related to ideas of Minkowski.

\begin{proposition}\label{push-pull}
Let  $(G,\p)$ be a bar framework whose points affinely span  $\mathbb{R}^d$ and $\p'$ be a non-trivial infinitesimal flex  of $(G,\p)$.
Then $(G,\p+\p')$ and $(G,\p-\p')$ have corresponding bar lengths the same, but they are not congruent. 
\end{proposition} 
\begin{proof} 
Suppose $ij\in E(G)$. Then 
$(p_i-p_j)\cdot (p'_i-p'_j)=0$, and so 
\begin{equation}
  \begin{aligned}
    \|(p_i+p'_i)-(p_j+p'_j)\|^2=
    \|(p_i-p_j)+(p'_i-p'_j)\|^2=\\
    (p_i-p_j)^2 +(p'_i-p'_j)^2 + 2(p_i-p_j)\cdot (p'_i-p'_j)=\nonumber\\  
    (p_i-p_j)^2 +(p'_i-p'_j)^2 - 2(p_i-p_j)\cdot (p'_i-p'_j)=\nonumber\\  
      \|(p_i-p_j)-(p'_i-p'_j)\|^2=\|(p_i-p'_i)-(p_j-p'_j)\|^2,
  \end{aligned}
\end{equation}
Hence $(G,\p+\p')$ and $(G,\p-\p')$ have corresponding bar lengths the same.  When $\p'$ is non-trivial, and the affine span of $\p$ is all of $\mathbb{R}^d$, then for some pair of vertices  $i,j\in V(G)$, $(p_i-p_j)\cdot (p'_i-p'_j)\ne 0$, and the above calculation shows that $(G,\p+\p')$ and $(G,\p-\p')$ are not congruent, since the middle equality, above, is violated.    
\end{proof}
  
Proposition \ref{push-pull} is illustrated in  Figure \ref{fig:Jackson}.
Notice that, due to the symmetry involved, the two non-congruent convex braced polygons on the top are mirror images of each other, but that congruence does not preserve the vertex labeling.  The mirror symmetry can be broken with different examples.

\begin{figure}[H]
\centering
\includegraphics[scale=0.45]{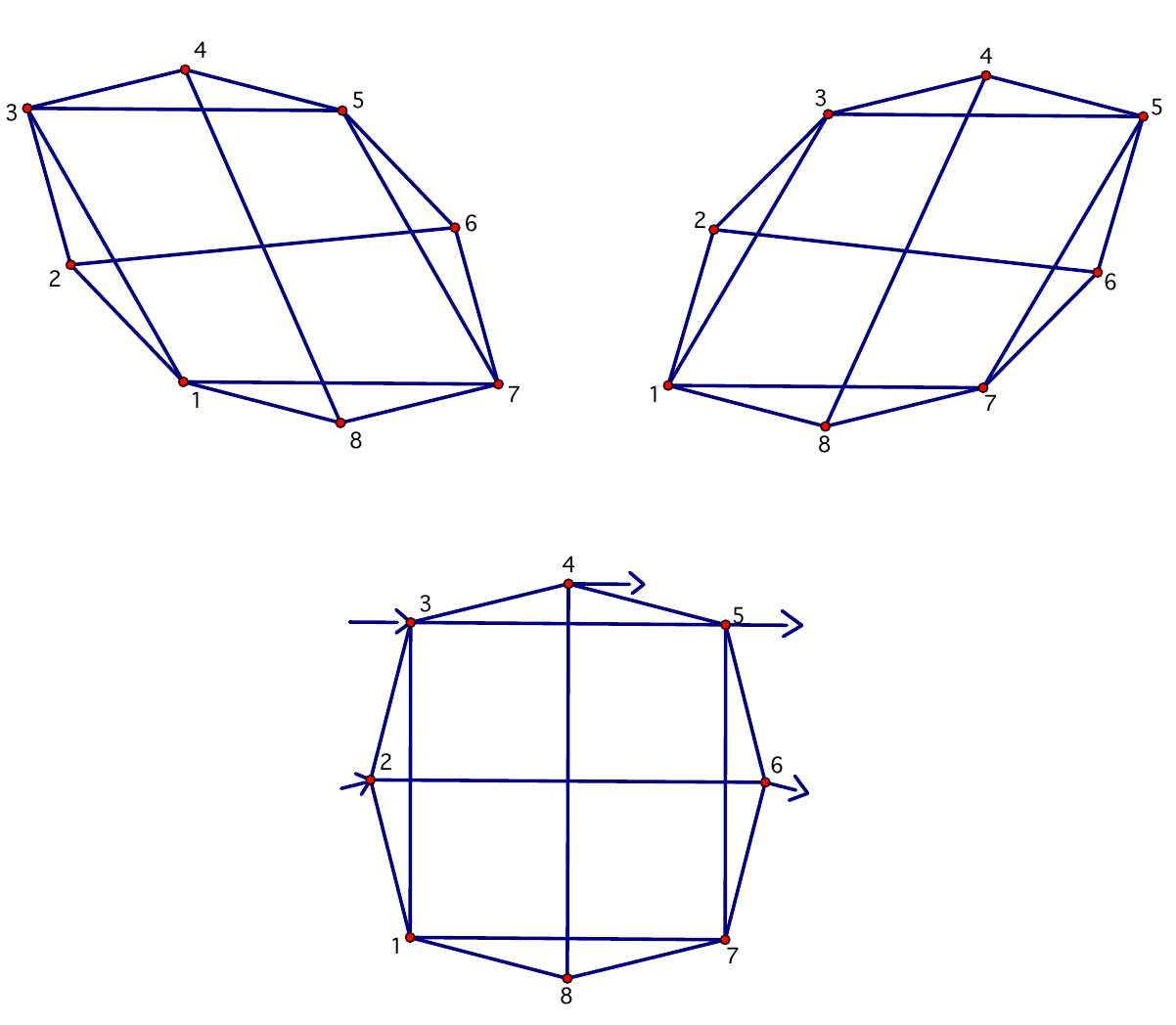}
\captionsetup{labelsep=colon,margin=1.3cm}
\caption{The two convex braced polygons on the top have corresponding bar lengths the same, but are not congruent. The bottom convex braced polygon has an infinitesimal flex $\p'$, where each non-zero $p'_i$ is indicated by a small arrow. This gives rise to the two convex braced polygons  on the top.} \label{fig:Jackson}
\end{figure}

The converse direction of Proposition~\ref{push-pull} also holds
and it is stated as follows.
The proof follows by reversing that of Proposition~\ref{push-pull}.
\begin{proposition}\label{reverse}
Let $(G,\bp)$ and $(G,\bq)$ be two non-congruent frameworks having corresponding bar lengths the same.
Then $\bp-\bq$ is a non-trivial infinitesimal motion of 
$(G,\frac{\bp+\bq}{2})$.
\end{proposition}

Now let us move to the proof of Theorem~\ref{convex}.
Proposition~\ref{push-pull} implies that,
if some braced strictly convex polygon $(G,\bp)$ has a non-trivial infinitesimal flex $\bp'$,
then $(G, \bp+\varepsilon \bp')$ is not convexly rigid for a sufficiently small $\varepsilon>0$. This gives one direction of Theorem~\ref{convex}.
We will verify the converse direction by using Proposition~\ref{reverse},
but this step requires an additional idea  because $(G,\frac{\bp+\bq}{2})$ may not be a strictly convex braced polygon even if $(G,\bp)$ and $(G,\bq)$ are.
See Figure~\ref{fig:parallel} for an example.
Our goal is to show the existance of  an appropriate rotation of $(G,\bq)$ such that  $(G,\frac{\bp+\bq}{2})$ is a strictly convex braced  polygon after rotation.

Suppose that $(G,\bp)$ and $(G,\bq)$ are strictly convex  braced polygons having corresponding bar lengths the same,
and denote $\bp=(p_1,\dots,p_n)$ and $\bq=(q_1,\dots, q_n)$.
We prepare a reference ray to be the horizontal ray to the left direction,
and for each $i=1,\dots,n$ let $\theta_i$ (resp.~$\phi_i$) be the counterclockwise angle from the reference ray to $p_{i+1}-p_i$ (resp.~$q_{i+1}-q_i$). 
By rotating $(G,\bp)$ (resp.~$(G,\bq)$), we may assume that $\theta_1=\phi_1=0$. See Figure~\ref{fig:angles}.
We can further extend the definition of $\theta_i$ (and similarly $\psi_i$) over any $i\in \mathbb{Z}$ by choosing $r,t\in \mathbb{Z}$ such that $i=tn+r$ and $1\leq r\leq n$, and then setting $\theta_i:=\theta_r+2\pi t$.

\begin{figure}[t]
\centering
\includegraphics[scale=0.9]{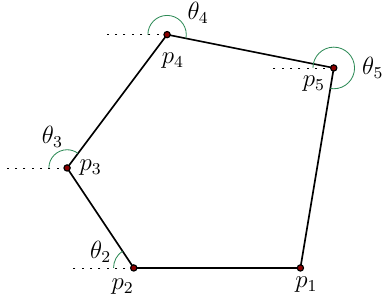}
\caption{Angles in a strictly convex polygon.}
\label{fig:angles}
\end{figure}

By strict convexity, we have 
$\theta_i<\theta_{i+1}<\theta_i+\pi$ for each $i$,
and this inequality is indeed a necessary and sufficient condition for $(G,\bp)$ to be strictly convex.

Consider $(G,\frac{\bp+\bq}{2})$
and define  the corresponding angles $\psi_i \ (i\in \mathbb{Z})$ for $(G,\frac{\bp+\bq}{2})$ by the same definition.
Then $\psi_i=\frac{\theta_i+\phi_i}{2}$ if 
$|\theta_i-\phi_i|$ is less than $\pi$.
(If $|\theta_i-\phi_i|$ is more than $\pi$,
then the direction of $\frac{p_{i+1}+q_{i+1}}{2}-\frac{p_i+q_i}{2}$ becomes opposite to the angle we are measuring between $p_{i+1}-p_i$ and $q_{i+1}-q_i$.
So $\psi_i$ becomes $\frac{\theta_i+\phi_i}{2}-\pi$. See Figure~\ref{fig:angle-wrong}.)
This implies that, if we can guarantee 
\begin{equation}
\label{eq:condition_for_covexity}
|\theta_i-\phi_i|<\pi \quad \text{for all $i$},
\end{equation}
then, for all $i$, 
$\psi_i=\frac{\theta_i+\phi_i}{2}$ holds 
and $\psi_i<\psi_{i+1}<\psi_i+\pi$ follows 
from $\theta_i<\theta_{i+1}<\theta_i+\pi$
and $\phi_i<\phi_{i+1}<\phi_i+\pi$.
Thus, $(G,\frac{\bp+\bq}{2})$ will be strictly convex.

Our next lemma shows that we can always achieve (\ref{eq:condition_for_covexity}) by rotating $(G,\bq)$.
We say that a function $f:\mathbb{R}\rightarrow \mathbb{R}$ is a {\em monotone step function} if it is a piecewise constant, upper-semi-continuous function, and non-decreasing.
\begin{lem}\label{lem:adjust}
Let $f,g:\mathbb{R} \rightarrow \mathbb{R}$ be monotone step functions.
Suppose that $(f,g)$ is periodic in the sense that, for some positive numbers $l$ and $C$, 
$f(x+l)=f(x)+C$ and  $g(x+l)=g(x)+C$ for all $x\in \mathbb{R}$.  
Suppose further that $f(x)\in [0,C)$ and 
$g(x)\in [0,C)$ for all $x\in [0,l)$.
Then there is a 
constant $\alpha$ such that $|f(x)-g(x)-\alpha| < \frac{C}{2}$ for all $x\in \mathbb{R}$. 
\end{lem}
\begin{proof} 
By the assumption on $f$ and $g$, 
$f-g$ is a periodic piecewise constant function with period $l$.
Then $f-g$ takes only finitely many values so  has a minimizer $m$ and a maximizer $M$.
By the periodicity of $f-g$, we may assume $M\in [m, m+l)$.
Then, 
\begin{equation}\label{eq:period}
0\leq f(M)-f(m)\leq C \mbox{ and } 0\leq g(M)-g(m)\leq C.
\end{equation}
Since $g(m)<C$ and $g(l)\geq C$,
if $g(m)=g(M)$, then $M<l$ holds and hence $f(M)<C$ holds.
In particular, by $f(m)\geq 0$, $f(M)-f(m)=C$ and $g(M)=g(m)$ cannot occur simultaneously.
Therefore, we have
\begin{align*}
0&\leq f(x)-g(x)-(f(m)-g(m)) \\
&\leq f(M)-g(M)-(f(m)-g(m)) \\
&=f(M)-f(m)-(g(M)-g(m)) \\
&< C-0=C,
\end{align*}
where the first inequality follows from the definition of $m$,
the second inequality follows from the definition of $M$, 
and the fourth strict inequality follows from 
(\ref{eq:period}) and the fact that $f(M)-f(m)=C$
and $g(M)=g(m)$ cannot occur simultaneously.
Let $F(x)=f(x)-f(m)$ and $G(x)=g(x)-g(m)$. The above inequality gives
\begin{equation}\label{eq:period2}
0\leq F(x)-G(x)\leq F(M)-G(M)< C
\end{equation}
for any $x\in\mathbb{R}$.
We set
\[
\alpha=f(m)-g(m)+\frac{F(M)-G(M)}{2}.
\]
Then, for any $x$,
$f(x)-g(x)-\alpha=F(x)-G(x)-\frac{F(M)-G(M)}{2}$.
Hence, by (\ref{eq:period2}),
$|f(x)-g(x)-\alpha|\leq |\frac{F(M)-G(M)}{2}|<\frac{C}{2}$.
\end{proof}



\begin{figure}[t]
\centering
\includegraphics[scale=0.7]{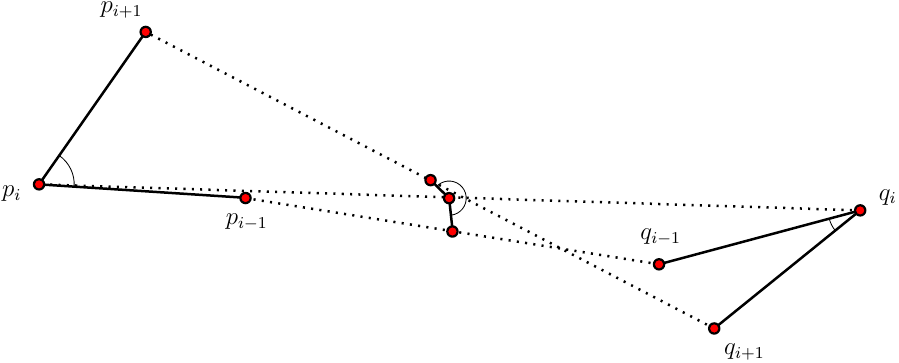}
\captionsetup{labelsep=colon,margin=1.3cm}
\caption{
This shows two convex angles such that their average is not convex.
} \label{fig:angle-wrong}
\end{figure}

\medskip

We are now ready to prove Theorem~\ref{convex}.

\medskip

 {\bf Proof of Theorem \ref{convex}:} We have already seen that, if $G$ has a strictly convex realisation with a non-trivial infinitesimal flex, then $G$ has two non-congruent strictly convex realisations in which corresponding bars have the same edge lengths (by Proposition~\ref{push-pull}).  
 
 To see the converse direction,  we assume that there are two non-congruent strictly convex braced polygons $(G,\p)$ and $(G,\q)$ with corresponding bars of the same length, both oriented counter-clockwise. 
 For $x\in \mathbb{R}$, let $p_x$ be the point on the boundary polygon of $(G,\bp)$
 whose distance from $p_1$ along the polygon in the counter-clockwise direction is equal to $x$.
 Consider $\theta_i$ as defined before Lemma~\ref{lem:adjust} based on $\bp$,
 and define $f:\mathbb{R}\rightarrow \mathbb{R}$
 by $f(x)=\theta_i$ for $x\in \mathbb{R}$ with $p_x\in [p_i,p_{i+1})$.
 Similarly, $g:\mathbb{R}\rightarrow \mathbb{R}$ is defined based on $\bq$.
 Since $(G,\bp)$ and $(G,\bq)$ have corresponding bar lengths the same, 
 $f$ and $g$ are monotone step functions 
 and they are periodic with $f(x+l)=f(x)+2\pi$ and $g(x+l)=g(x)+2\pi$ for any $x\in \mathbb{R}$, where $l$ is the circumference of the boundary polygons of $(G,\bp)$ and $(G,\bq)$.
 By Lemma~\ref{lem:adjust}, there is an $\alpha$ such that 
 $|f(x)-g(x)-\alpha|<\pi$.
 This in turn implies that, if $(G,\bq')$ is obtained by  rotating $(G,\bq)$ by $\alpha$,
 (\ref{eq:condition_for_covexity}) holds between $(G,\bp)$ and $(G,\bq')$,
 and $(G,\frac{\bp+\bq'}{2})$ forms a braced strictly convex polygon.
 By Proposition~\ref{reverse}, $(G,\frac{\bp+\bq'}{2})$ is not infinitesimally rigid.
 This completes the proof.
 \qed

 \begin{rem}\label{rem:parallel}Even when one strictly convex braced polygon $(G,\p)$ has one of its edges parallel to the corresponding edge of $(G,\q)$, the average $(G,(\p+\q)/2)$ may not be convex as seen in Figure \ref{fig:parallel}.
 \end{rem}
\begin{figure}[H]
\centering
\includegraphics[scale=0.4]{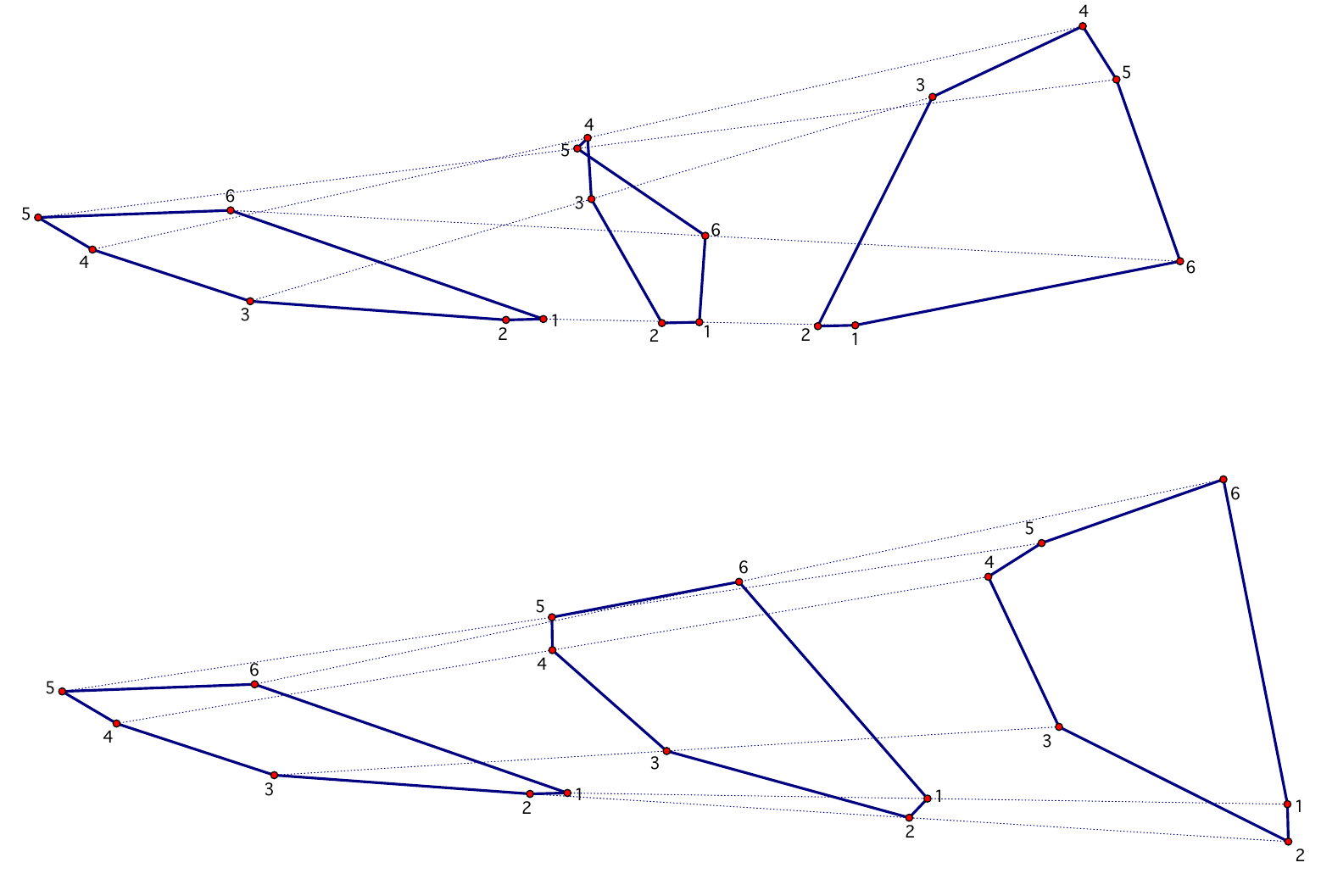}
\captionsetup{labelsep=colon,margin=1.3cm}
\caption{ The top figure shows two strictly convex braced polygons with corresponding sides the same length, and with one pair $\{1,2\}$ of corresponding sides that are parallel and have the same direction, but their average is not a convex polygon.  The bottom figure shows the same two strictly convex braced polygons but with one rotated so that their average is convex.  The dotted line segment between corresponding vertices have the average vertices at their midpoint. } \label{fig:parallel}
\end{figure}
 \begin{rem}\label{rem:unequal}Figure \ref{fig:unequal} shows two frameworks consisting of two bars $b_1,b_2$ and $c_1,c_2$. Bars $b_1,c_1$ have equal length but the other pair $b_2, c_2$ do not have equal length. Both $b_1,b_2$ and $c_1,c_2$ have the same orientation.  However, the average pair $(b_1+c_1)/2,(b_2+c_2)/2$ has the opposite orientation.  So the sign of the angle between  successive bars can change when corresponding bars do not have the same length.
\end{rem}
\begin{figure}[H]
\centering
\includegraphics[scale=0.3]{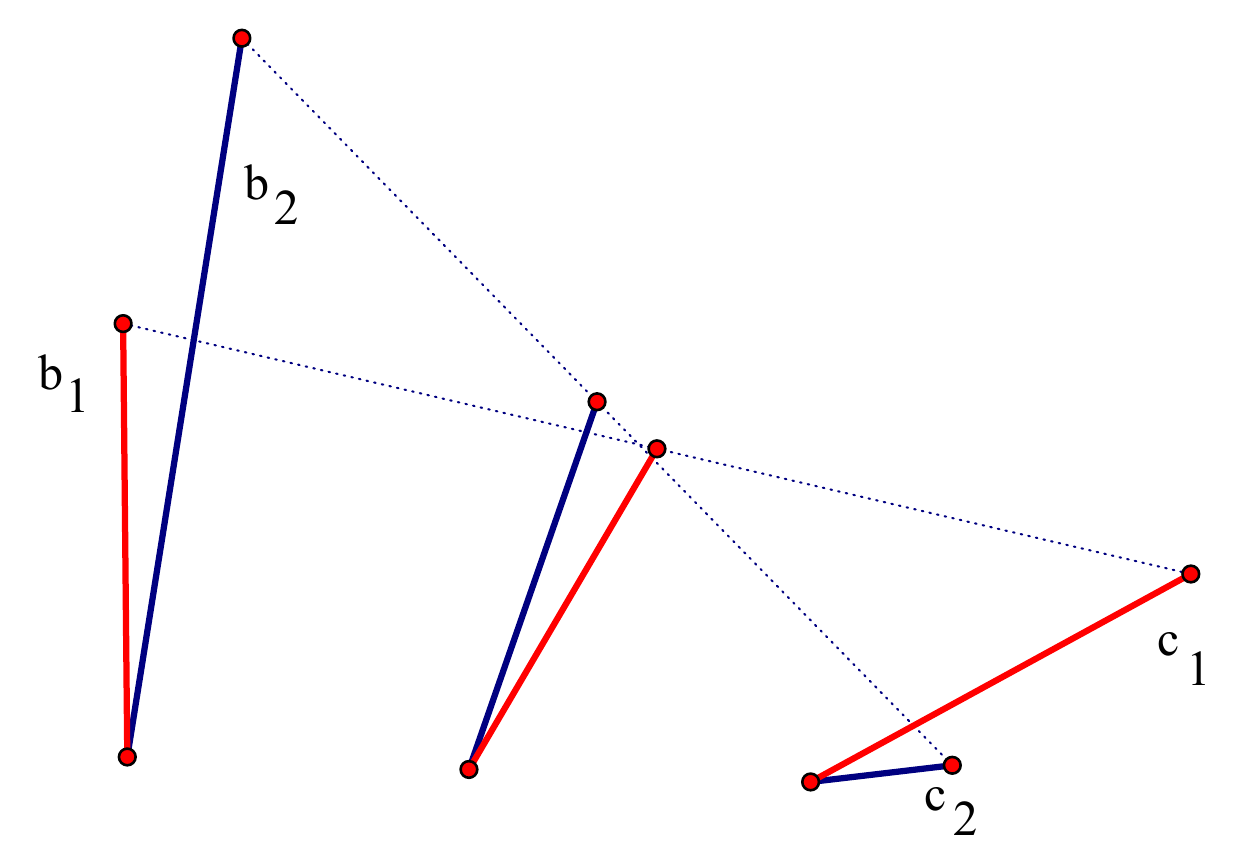}
\captionsetup{labelsep=colon,margin=1.3cm}
\caption{Two frameworks, where not all corresponding pairs of bars have equal length, and the average configuration has the opposite orientation} \label{fig:unequal}
\end{figure}
\begin{rem}\label{rem:non-global}Figure \ref{fig:non-global} shows two non-congruent frameworks that have corresponding bars of the same length.  Hence neither framework is globally rigid. The one on the right is a strictly convex braced polygon that is infinitesimally rigid in all strictly convex realizations and so is convexly rigid by Theorem \ref{convex}. Thus the left framework is necessarily not strictly convex. The underlying graph of these frameworks is globally rigid in the plane at all generic configurations.
\end{rem}
\begin{figure}[H]
\centering
\includegraphics[scale=0.4]{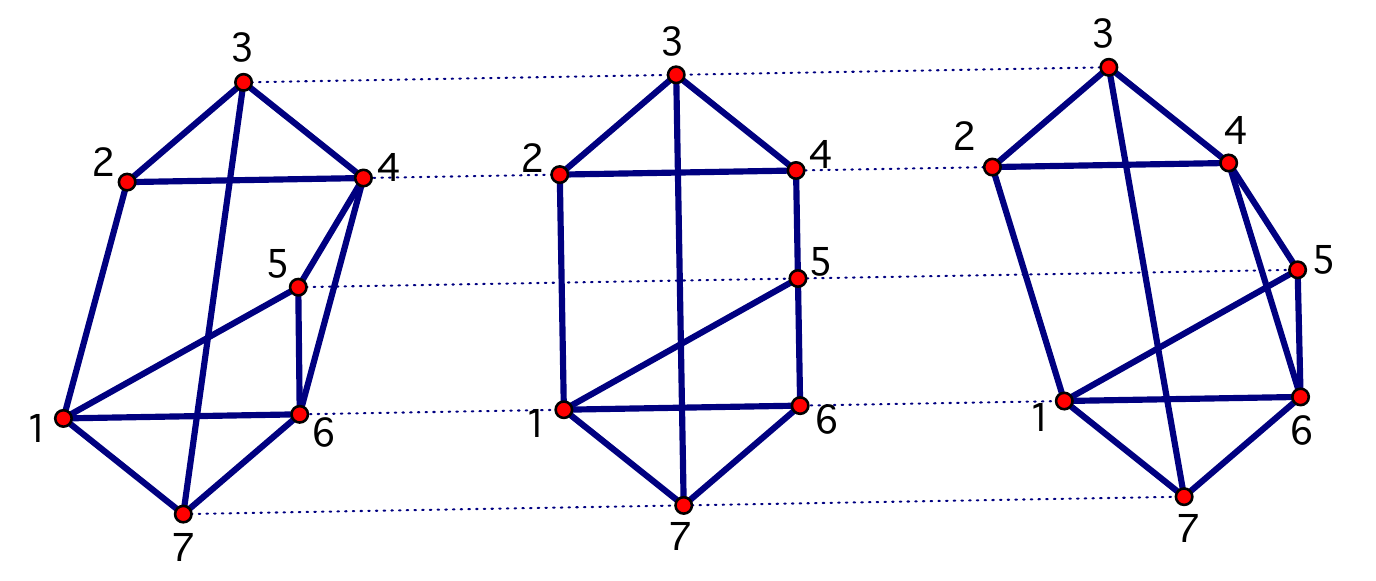}
\captionsetup{labelsep=colon,margin=1.3cm}
\caption{This shows two non-congruent bar frameworks for the same graph that have corresponding bars the same length.  The one on the right is a braced strictly convex polygon. The one in the middle is an average of the two other non-congruent configurations.  Note that the middle framework is convex, but not strictly convex, and is necessarily not infinitesimally rigid in the plane. The points $1,7,6,5$ are fixed in all three realizations. } \label{fig:non-global}
\end{figure}


\section{Braced Polygons with Proper Stresses}\label{Sect:Braced}
In this section, we shall strengthen Theorem~\ref{convex} by restricting our attention to special families of braced polygons.
Our discussion is motivated by a result by Geleji and Jordan~\\\cite{robust}, which characterizes the global rigidity of all strictly convex plane realizations of a braced polygon graph which is a circuit in the generic 2-dimensional rigidity matroid.
We shall develop a novel analysis of the rigidity of braced polygon graphs by focusing on minimal 3-connectivity.

\subsection{Proper Stresses and Super Stability}\label{subsec:3-1}
Our tool to prove global rigidity (or even universal rigidity) is
super stability, which is based on the concept of an equilibrium stress.

\begin{defn}
For a configuration $(G,\p)$, a \textit{stress} is a function $\omega:E\to \mathbb{R}$ that
maps the set of edges of $G$ to the real numbers. The stress is in \textit{equilibrium} if $\sum_{j} \omega_{ij}(\mathbf{p}_i-\mathbf{p}_j)=\mathbf{0}$ for every vertex $i$, taking $\omega_{ij}=0$ when $ij\not\in E$.   The \textit{stress matrix} $\Omega$ of a stress $\omega$ is an $n\times n$ symmetric matrix defined as $\Omega_{ij}=-\omega_{ji}$ for each edge $e_{ij}$, $\Omega_{ij}=0$ if $i\not=j$ and there is no edge between $v_i$ and $v_j$, and that each row sums to $0$. 

We say that an equilibrium stress  for a convex braced polygon is {\em proper} if it  is positive on the boundary polygon and negative on the interior edges, and that the convex braced polygon is {\em properly stressed} if it has a proper stress.
\end{defn}

\begin{defn}
    A framework $(G,\p)$ of dimension $d$ with stress matrix $\Omega$ is \textit{super stable} if $\Omega$ is positive semidefinite with rank $n-d-1$, and all affine transformations that preserve the length of all edges are trivial motions of the framework. 
\end{defn}

A basic theorem in \cite[Theorem 5.14.1]{book} states that super stable frameworks are universally rigid. The reverse is not true. See Figure \ref{fig:visual} in the plane, and see \cite{iterative} for the general situation.

For braced polygons, the following sufficient condition for super stability is useful.

\begin{thm}[Connelly~\cite{polygon}]\label{thm:connelly}
Suppose a strictly convex braced polygon $(G,\bp)$ has  a proper stress.
Then it is super stable for the corresponding stress matrix.
\end{thm}

The following easy observation that follows from convexity is  another key tool.

\begin{proposition}\label{prop:convex_stress}
Let $(G,\bp)$ be a strictly convex braced polygon 
and suppose that $(G,\bp)$ has a non-zero equilibrium stress $\omega$ that is non-positive  on all the interior edges.
Then $\omega$ is positive on the boundary edges.
\end{proposition}
\begin{proof}
At each vertex, consider the three vectors, centered at the vertex; the sum of the interior edges weighted by their non-positive stresses, and the two boundary edges weighted by their stresses.  The equilibrium condition implies that the sum of these three vectors is $0$.  If just one of these vectors is $0$, it implies that the other two are negatives of each other, and this contradicts the strict convexity at that vertex.  Similarly, if two of the vectors are $0$, so is the third and they are all $0$.  Then the local strict convexity at the vertex implies that when all three are non-zero, the boundary vertex stresses are both positive by the independence of the three vectors.  

So either all the stresses are $0$, or all the boundary stresses are positive, and there is at least one interior stress that is negative at each vertex. 
\end{proof}


\subsection{Minimally 3-connected Braced Polygons with Proper Stresses}\label{subsec:3-2}
In view of Theorem~\ref{thm:connelly}, the next natural question is to identify the strictly convex braced polygons which have a proper stress. 
We give a positive answer for  the family of strictly convex realisations of minimally 3-connected braced polygon graphs, whose  formal definition is given as follows.

\begin{defn}
A braced polygon graph is \textit{minimally $k$-connected} if it is $k$-connected, and removing any brace (interior edge) would cause it to be no longer $k$-connected. 
\end{defn}

The following theorem is a main technical observation.

\begin{thm}\label{stressExist}
    Every 3-connected braced polygon graph $P$ has a strictly convex realization that is properly stressed. In addition, if $B$ is a set of braces such that $P-B$ is 3-connected, then there exists a strictly convex realization of $P$ such that $P-B$ is properly stressed and all braces in $B$ have positive stresses. 
\end{thm}

Proving Theorem \ref{stressExist} requires several additional ingredients that will be proved individually. First, we show there is an inductive procedure to produce any 3-connected braced polygon graph. Next, we show that if a convex but not strictly convex braced polygon is properly stressed, then there exists a nearby strictly convex configuration with a proper stress by minimizing a certain energy function. 

\begin{proposition}\label{operation}
    Every 3-connected braced polygon graph $P$ with at least 4 vertices can be constructed from $K_4$ by adding one brace at a time using the following operations:
\begin{itemize}
    \item remove a boundary edge $(i,j)$, add a new vertex $k$ and create two new boundary edges $(i,k)$ and $(j,k)$, and connect $k$ to another vertex. 

    \item remove two boundary edges (possibly with one overlapping vertex) $(i_1,j_1)$ and $(i_2,j_2)$, add four boundary edges $(i_1,k_1),(j_1,k_1),(i_2,k_2),(j_2,k_2)$ and a brace $(k_1,k_2)$.

    \item add a brace between two existing vertices. 
\end{itemize}

In addition, 3-connectivity is preserved at every step of the above construction and, if $P$ is minimally 3-connected, then it can be constructed by only using the first 2 operations. 
\end{proposition}
\begin{proof}
    Since $P$ is 3-connected, every brace must be crossed by another brace in the interior of the polygon and all braces must be connected by crossing points. Take 2 braces that cross each other and identify them with the 2 braces in $Q=K_4$ together with 4 vertices. 
    
\begin{figure}[H]
\centering
\includegraphics[scale=0.6]{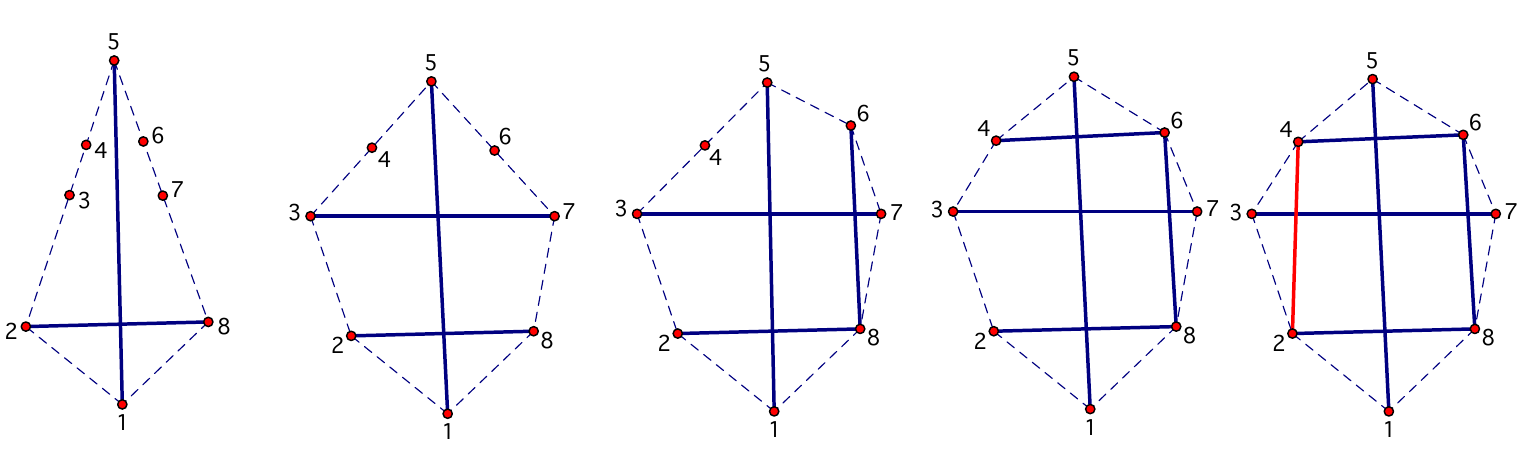}
\captionsetup{labelsep=colon,margin=1.3cm}
\caption{The degree 2 vertices are shown only to make the vertex correspondence clearer. They don't exist in our construction until they are added through a new brace. In each step, one brace that crosses existing braces is added, so the new graph must stay 3 vertex connected.} \label{fig:Sequence}
\end{figure}

    If $P\not= Q$, then there exists a brace $(k_1,k_2)$ from $P$ that crosses a brace identified in $Q$. Otherwise, the braces in $P$ identified in $Q$ and the remaining braces never cross each other, contradicting $P$ being 3-connected. Depending on the number of vertices in $k_1$ and $k_2$ identified in $Q$, choose one of the 3 operations so that the new edge $(k_1',k_2')$ has the same crossing pattern with the existing braces in $Q$ as $(k_1,k_2)$ with the corresponding braces in $P$. This procedure is demonstrated in Figure \ref{fig:Sequence}. 

    To prove the last statement, we show that if we do the last operation, then the added edge will never become necessary for the 3-connectivity. Therefore, $P$ is not minimal if the last operation is used in the process. 

    Suppose that we can add $(k_1,k_2)$ to $Q$ by the last operation, then the new graph is immediately no longer minimal. Since $Q$ is 3-connected without $(k_1,k_2)$, there must be a path $L$  from $k_1$ to $k_2$ in the interior of the polygon joined by the crossings without $(k_1,k_2)$. Therefore, any edge added later that crosses $(k_1,k_2)$ must also cross the path $L$. As a result, all the braces are still connected by crossings, even if we remove $(k_1,k_2)$.
\end{proof}

\begin{proposition}\label{energyfunction}
    There exists a differentiable non-negative strictly monotonically increasing real function $f$ defined on $[0,\infty)$ such that given $a,m,\delta,\omega>0$, we have $f(a)=\frac{1}{m}$, $f'(a)=\omega$, and $f(a+\delta)>1$. Similarly, there exists a differentiable positive strictly monotonically decreasing function $g$ defined on $(0,\infty)$ such that given $a,m,\delta>0$ and $\omega<0$, we have $g(a)=\frac{1}{m}$, $g'(a)=\omega$, and $g(a-\delta)>1$. 
\end{proposition}

\begin{proof}
    To construct $f$, consider the graph of the derivative $f'(x)$. Draw a curve in the first quadrant from $(0,0)$ to $(a,\omega)$ with an area above the interval $(0,a)$ being $\frac{1}{m}$, then connect $(a,\omega)$ to $(a+\delta,\frac{2}{\delta})$ with a straight line. The remaining $f'$ only has to be continuous and positive. Let $f(x):=\int_{0}^{x}f'(y)dy$. The construction of $f'$ is shown in Figure \ref{fig:fprime}. 

    \begin{figure}[H]
    \centering
    \includegraphics[scale=0.5]{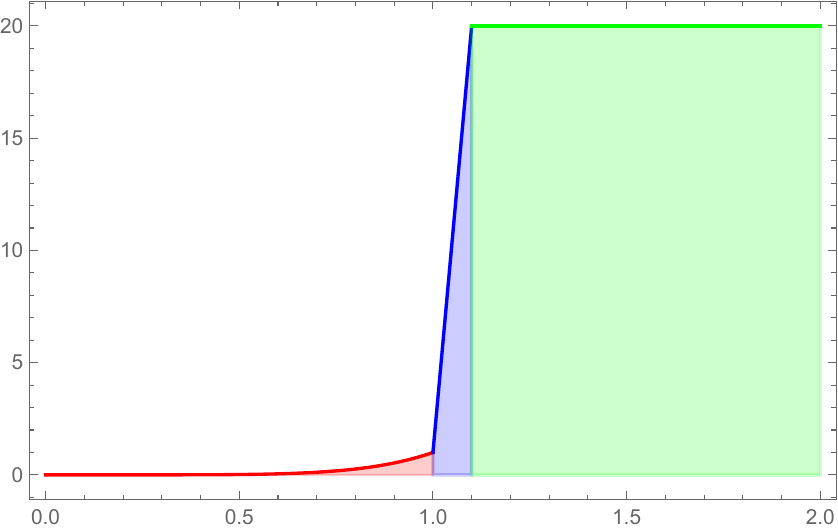}
    \captionsetup{labelsep=colon,margin=1.3cm}
    \caption{This is the derivative of $f$. The red region has area $\frac{1}{m}$. The blue region has area at least $1$. The rest of the function only has to be positive and continuous.} \label{fig:fprime}
    \end{figure}

    To construct $g$, consider $g'(x)$. Connect $(a,\omega)$ to $(a-\delta,-\frac{2}{\delta})$ with a straight line. $g'(x)$ can be any negative curve with an area above the interval $(a,\infty)$ converging to $\frac{1}{m}$ as $x\to\infty$. The remaining part of $g'$ just has to be continuous and negative. Now, let $g(x):=\int_{\infty}^{x}g'(y)dy$. 
\end{proof}

\begin{figure}[H]
\centering
\includegraphics[scale=0.35]{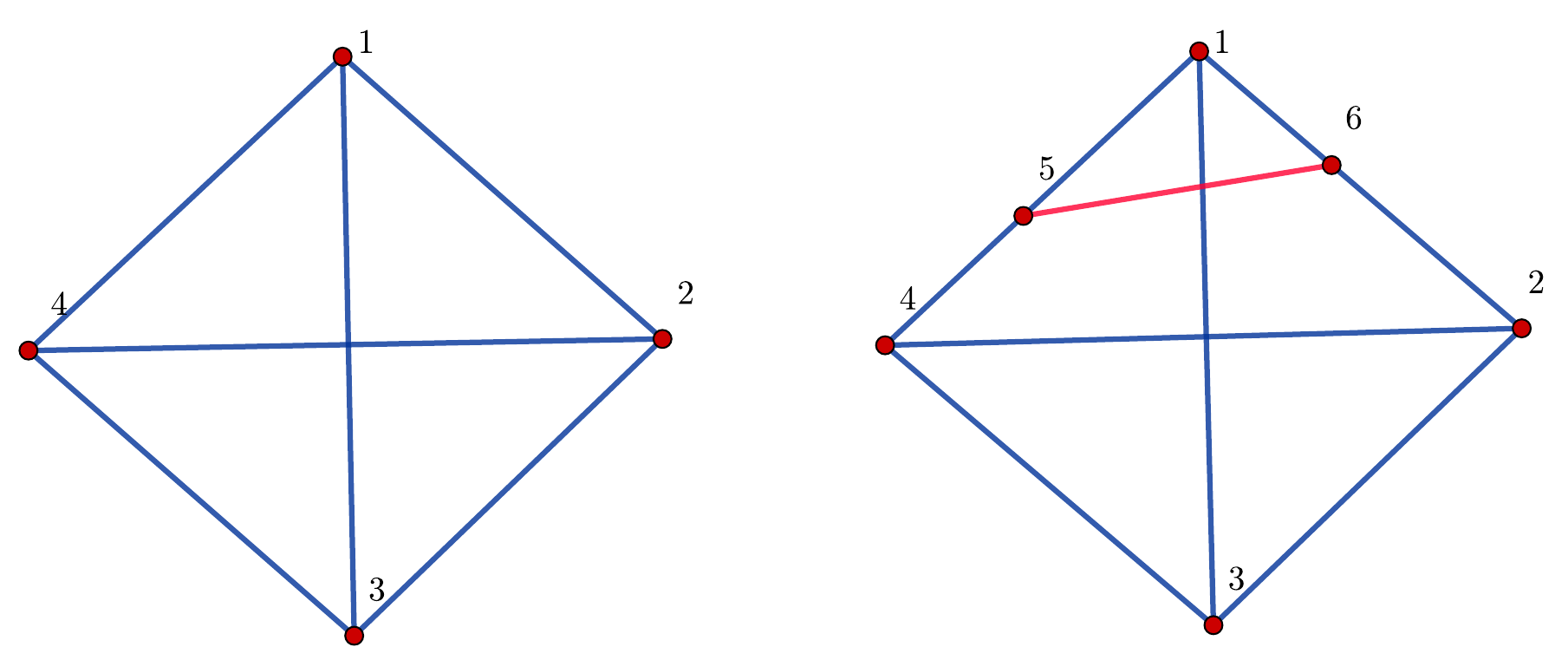}
\captionsetup{labelsep=colon,margin=1.3cm}
\caption{brace $(5,6)$ is added by the second operation defined in Proposition \ref{operation}} \label{fig:addedge}
\end{figure}

\medskip
Now we are ready to prove Theorem \ref{stressExist}.

\medskip

\noindent
{\bf Proof of Theorem \ref{stressExist}:}
    We proceed by induction. Since $P$ is minimally 3-connected, it can be constructed using 2 operations according to Proposition \ref{operation}. The base case, $K_4$, can be placed on a square to generate a stress with the desired signs. For the induction step, we need to show that if a braced polygon $P$ is properly stressed in a strictly convex configuration, then the braced polygon $Q$ after adding a brace $(k_1,k_2)$ through an operation can also be properly stressed in a strictly convex configuration. Figure \ref{fig:addedge} is a case where the brace $(5,6)$ is added.

    The same embedding as $P$ with new vertices placed on existing edges (shown in Figure \ref{fig:addedge}) immediately gives an equilibrium stress $\omega$ on $Q$ where all forces (stress times length) on existing edges remain unchanged and the stress on $(k_1,k_2)$ is $0$. We call this configuration $(Q,\mathbf{p}_0)$. There are two problems that need to be solved: $(Q,\mathbf{p}_0)$ is no longer strictly convex and the stress is not strictly negative on the new brace. 

    Let $p_i$ be the coordinate of the vertex $i$ in $(Q,\mathbf{p})$, $l(\mathbf{p})$ be the length function maps a configuration $\mathbf{p}$ to $\mathbb{R}^m$, $l_{ij}(\mathbf{p})$ be the distance from $p_i$ to $p_j$, $m$ be the number of edges in $Q$, $\epsilon,\delta>0$ be real numbers, $\omega_{ij}$ be the stress on the edge $(i,j)$ with the only exception $\omega_{k_1k_2}=-\epsilon$. By Proposition \ref{energyfunction}, we can find monotonic and differentiable energy functions $E_{ij}(l^2_{ij})$ for each edge $(i,j)$ such that $E_{ij}=\frac{1}{m}$ at $\mathbf{p}_0$, $E_{ij}'(l_{ij}^2)=\omega_{ij}$, and $E_{ij}>1$ if $l_{ij}^2$ increases by $\delta$ for $\omega_{ij}>0$ or decreases by $\delta$ for $\omega_{ij}<0$. Let $E(Q,\mathbf{p})=\sum_{ij\in Q}E_{ij}(l_{ij}^2(\mathbf{p}))$. This energy function has a few good properties:
    \begin{itemize}
        \item At the given configuration $(Q,\mathbf{p}_0)$, $E(Q,\mathbf{p}_0)=\sum_{ij\in Q}\frac{1}{m}=1$
        \item Fixing a vertex $i$, $\frac{\partial }{\partial x_i}E_{ij}(l_{ij}^2)=2E_{ij}'(l^2_{ij})(x_i-x_j)$, hence a critical point of $E$ gives an equilibrium stress with value $E_{ij}'(l^2_{ij})$ on edge $(i,j)$
        \item Pulling $k_1$ and $k_2$ outwards infinitesimally in direction orthogonal to the boundary will decrease $E$, hence there exists a configurations nearby with $E<1$. 
    \end{itemize}
    
    Now we make the following observation: if $\delta$ is sufficiently small, then $l_{ij}(\mathbf{p})$ must stay close to  $l_{ij}(\mathbf{p}_0)$ in order to keep $E(Q,\mathbf{p})<1$. Suppose that a boundary edge $(i,j)$  decreases in length by $\xi$. Because the configuration $(Q,\mathbf{p}_0)$ is super stable, $\xi$ must converge to $0$ as $\delta$ goes to $0$. Therefore, we find that $l_{ij}(\mathbf{p})\to l_{ij}(\mathbf{p}_0)$ for all $\mathbf{p}$ such that $E(Q,\mathbf{p})<1$ as $\delta\to 0$. The argument for braces is similar. This implies that if $\delta$ is small enough, there is a critical point $l(\mathbf{q})$ of $E$ in the small neighborhood of $l(\mathbf{p}_0)$ giving the minimum of $E$. 

    Next, we claim that if $l(\mathbf{q})$ is close to $l(\mathbf{p}_0)$, then the configuration $\mathbf{q}$ must be close to $\mathbf{p}_0$ (up to Euclidean isometry). Consider the sequence $\delta_i=\frac{1}{i}$. Without loss of generality, we fix a vertex and the direction of an edge connected to it. Let $\{\mathbf{q}_i\}=\mathbf{q}_1,\mathbf{q}_2,...$ be a sequence of possible configurations. Since the lengths are all bounded, a configuration will always stay in a compact neighborhood. Therefore, some subsequence of $\{\mathbf{q}_i\}$ converges to a point $\Tilde{\mathbf{p}}$. If $\Tilde{\mathbf{p}}\not=\mathbf{p}_0$, then we have another configuration with $l(\Tilde{\mathbf{p}})=l(\mathbf{p}_0)$, which contradicts super stability. 

    Finally, for a chosen $\epsilon$, pick a small enough $\delta$, and let a critical point of $E$ near $\mathbf{p}_0$ be $\mathbf{q}$. $E'(l^2(\mathbf{q}))$ is an equilibrium stress. If $\delta$ is sufficiently small, no stress from $P$ can change sign because $E'$ is continuous. Notice that the stress on $(k_1,k_2)$ must be negative because our $E_{k_1 k_2}$ strictly decreases in Proposition \ref{energyfunction}. 

    To see that $\mathbf{q}$ is strictly convex, all vertices are in a small neighborhood of $\mathbf{p}$, therefore, every vertex that was strictly convex stays so with a sufficiently small $\delta$. For the newly added vertices, they have degree 3 with positive stress on the boundary and negative stress on a brace, so they must be strictly convex.

    By the same energy minimization argument, if an edge is not necessary for the 3 vertex connectivity, then it can have either positive or negative stress. We can pick $\epsilon$ to be either positive or negative in this case and a small $\delta$. A point in the small neighborhood of a strictly convex polygon must also be strictly convex. This concludes the last part of Theorem \ref{stressExist}. 
    \qed

\medskip

\subsection{Minimal 3-connectivity and circuits}
Minimally 3-connected braced polygon graphs enjoy the following special property.
\begin{lem}\label{lem:3con_2stresses}
Let $G$ be a minimally 3-connected braced polygon graph, $(G,\p)$ be a strictly convex realisation of $G$
and $\omega$ be a non-zero equilibrium stress of $(G,\bp)$ 
that is non-positive on the interior braces.
Then $\omega$ is a proper stress and  the space of equilibrium stresses of $(G,\bp)$ has dimension one.
\end{lem}
\begin{proof}
By Proposition~\ref{prop:convex_stress}, $\omega$ is positive on the boundary edges.
Suppose $\omega$ is zero on some interior brace.
Let $Z$ be the set of interior braces which have zero stress in $\omega$.
Then, $\omega$ is a proper stress of $(G-Z,\bp)$.
This implies from Theorem~\ref{thm:connelly} that  $G-Z$ is 3-connected.
Since $Z$ is non-empty,
this contradicts the minimal 3-connectivity of $G$.
Thus $\omega$ is a proper stress.

Suppose that the space of equilibrium stresses of $(G,\bp)$ has dimension at least two.
Let $\omega_1$ be a proper stress and $\omega_2$ be an arbitrary equilibrium stress that is not a scalar multiple of $\omega_1$.
We may assume that $\omega_1(e)$ and $\omega_2(e)$ have different signs on some edge $e$ (by taking the negative of $\omega_2$ if necessary).

Consider $\omega_t:=\omega_1+t\omega_2$ for a parameter $t$ starting at $t=0$, and increase $t$ continuously.
Let $\bar{t}$ be the smallest $t$ such that  $\omega_t(e)=0$ holds for some edge $e$.
Since $\omega_{\bar{t}}$ is non-positive on the interior braces,  the existence of $\omega_{\bar{t}}$ contradicts the former part of the statement.
\end{proof}

\medskip

Graphs whose edge sets are circuits in the generic $2$-dimensional rigidity matroid are important in the study of global rigidity in $\mathbb{R}^2$. We next consider braced polygon graphs which are generic rigidity circuits.
\begin{defn}
A graph $G$, with with $n$ vertices and $m$ edges, is called a \textit{generic rigidity circuit} if $m=2n-2$ and for every subgraph on $k$ vertices such that $2\leq k<n$, the number of edges is smaller than or equal to $2k-3$. 
A braced polygon graph which is a generic rigidity circuit is simply called a {\em braced polygonal circuit}.
\end{defn}

An important corollary of Theorem~\ref{stressExist} and Lemma~\ref{lem:3con_2stresses} is the following.

\begin{lem}\label{lem:3con_circuits}
Let $G$ be a minimally 3-connected braced polygon graph with $n$ vertices and $m$ edges.
If $m\geq 2n-2$, then $G$ is a generic rigidity circuit
and  every strictly convex braced polygon $(G,\bp)$ is properly stressed.
\end{lem}
\begin{proof}
By Theorem~\ref{stressExist}, 
there is a strictly convex framework $(G,\bq)$ having a proper stress.
By Lemma~\ref{lem:3con_2stresses}, the space of equilibrium stresses of $(G,\bq)$ has dimension one.
Therefore, by $m\geq 2n-2$, 
we must have $m=2n-2$ and the rigidity matrix has rank equal to $2n-3$.
Moreover, every edge has a non-zero stress in the proper stress, which means that  $G$ is a circuit.

We now show that every strictly convex polygon $(G,\bp)$ is properly stressed.
Let $G_1,\dots, G_s$ be the set of generically minimally rigid spanning subgraphs of $G$,
and let $N_i$ be an open subspace of the configuration space of strictly convex realizations of $G$ such that 
$G_i$ forms a minimally infinitesimally rigid framework $(G_i,\bp)$ for any $\bp\in N_i$.
Over $N_i$, the rigidity matrix can be solved with respect to the row vectors associated with $G_i$ (after an appropriate pin down), so the space of equilibrium stresses can be described as a rational function in the entries of $\bp$.

As we have seen above, $G$ has an infinitesimally rigid strictly convex realization $(G,\bp_0)$ having a proper stress $\omega_0$.
Consider any strictly convex framework $(G,\bp_1)$ of $G$.
Since the configuration space of a convex polygon with fixed  
graph is connected, we can continuously deform $(G,\bp_0)$ into $(G,\bp_1)$ 
by a continuous path $\bp_t\ (t\in [0,1])$ within the space of convex polygons. 
Suppose that $(G,\bp_t)$ is infinitesimally rigid for any $t\in [0,1]$, i.e., the continuous path is covered by $\bigcup_{i=1}^s N_i$.
Then there is a non-zero equilibrium stress $\omega_t$ of $(G,\bp_t)$ such that the stress varies continuously in $t$.
Then, by Lemma~\ref{lem:3con_2stresses} and the fact that $\omega_0$ is a proper stress, $\omega_t$ remains a proper stress for all $t\in [0,1]$.
In particular, $(G,\bp_1)$ is properly stressed.

Suppose $(G,\bp_t)$ is not infinitesimally rigid for some $t$.
Since $(G,\bp_0)$ is infinitesimally rigid, we have $t>0$.
We may assume that $(G,\bp_{t'})$ is infinitesimally rigid for any $t'\in [0,t)$ and is properly stressed (by the same argument as above).
We consider a sequence $\{\bp_{t_i}\}_{i=1,2.\dots}$ of point configurations that converges to $\bp_{t}$.
Let $\omega_{t_i}$ be a proper stress of $(G,\bp_{t_i})$ having unit norm. By the compactness of the unit sphere, a subsequence of $\{\omega_{t_i}\}_{i=1,2,\dots}$ converges to a nonzero $\omega$, which is an equilibrium stress of $(G,\bp_{t})$.
Since $\omega$ is non-positive on the interior braces, 
Lemma~\ref{lem:3con_2stresses} implies that $\omega$ is a proper stress.
However, since $G$ is a circuit and $(G,\bp_{t})$ is not infinitesimally rigid, the space of equilibrium stressed of $(G,\bp_{t})$ is at least two, contradicting Lemma~\ref{lem:3con_2stresses}.  
This completes the proof.
\end{proof}

\subsection{Main theorems}
We are now ready to give our main result.
\begin{thm}\label{thm:minimally}
Let $n\geq 4$.
The following are equivalent for a minimally 3-connected braced polygon graph $G$ with $n$ vertices.
\begin{itemize}
\item[(a)]  All strictly convex braced polygons $(G,\bp)$ are infinitesimally rigid.
\item[(b)] All strictly convex  braced polygons $(G,\bp)$ are convexly rigid.
\item[(c)]  All strictly convex  braced polygons $(G,\bp)$ are globally rigid.
\item[(d)] All strictly convex  braced polygons $(G,\bp)$ are super stable.
\item[(e)] All strictly convex braced polygons $(G,\bp)$ are properly stressed.
\item[(g)] $G$ is a generic rigidity circuit.
\item[(h)] $G$ has $n-2$ internal braces.
\end{itemize}
\end{thm}
\begin{proof}
The equivalence between (a) and (b) follows from 
Theorem~\ref{convex}.
(c) implies (b)  by definition,
and (d) implies (c) by Connelly's super stability theorem.
(e) implies (d) by Theorem~\ref{thm:connelly}.
(g) implies (e) by Lemma~\ref{lem:3con_circuits}.
Finally, Suppose (a) holds. 
By Theorem~\ref{stressExist},
$G$ has an infinitesimally rigid strictly convex framework $(G,\bp)$ with a proper stress.
Hence, $m\geq 2n-3+1=2n-2$.
By Lemma~\ref{lem:3con_circuits},
$G$ is a  generic rigidity circuit and (g) follows. Similarly, Lemma~\ref{lem:3con_circuits} shows that (g) and (h) are equivalent.
\end{proof}

\medskip

Geleji and Jord{\'a}n~\cite[Theorem 1.2]{robust} gave a criterion, known as the ``unique interval property"(See Appendix), that is equivalent to all convex configurations of a braced polygonal circuit having a proper stress. We show that this is also equivalent to minimal 3-connectedness.
\begin{thm}\label{m3p}
    Let $G$ be a braced polygon graph with $n$ vertices and $2n-2$ edges.
    Then, $G$ is minimally 3-connected if and only if all strictly convex braced polygons  $(G,\bp)$ are properly stressed.
\end{thm}
\begin{proof}
Suppose $G$ is minimally 3-connected.
By Lemma~\ref{lem:3con_circuits}, $G$ is a generic rigidity circuit.
By Theorem~\ref{thm:minimally}, all strictly convex braced polygon of $G$ are properly stressed.

To see the other direction, suppose that $G$ is not minimally 3-connected.
Then, pick a spanning minimally 3-connected subgraph $H$ of $G$.
By Theorem~\ref{stressExist}, there is a strictly convex framework $(H,\bp)$ which is properly stressed.
If $(G,\bp)$ is not infinitesimally rigid, then we are done
by Theorem~\ref{convex}.
So, assume $(G,\bp)$ is infinitesimally rigid.
Then the rank of the rigidity matrix of $(G,\bp)$ is $2n-3$
and the space of equilibrium stresses of $(G,\bp)$ is one dimensional by $m=2n-2$.
This in turn implies that
no edge in $E(G)\setminus E(H)$ can be stressed in $(G,\bp)$
(i.e., each edge in $E(G)\setminus E(H)$ is a coloop in the rigidity matroid of $(G,\bp)$).
In other words, $(G,\bp)$ is not properly stressed.
\end{proof}
\begin{remark}
    The condition of having $n-2$ braces guarantees the existence of a stress in any strictly convex configuration. Then the minimal 3-connectivity prevents the stress on any edge from changing its sign. Does having a stress in a minimally 3-connected convex braced polygon imply the existance of a proper stress in that particular configuration when there are not enough braces? The answer is negative, as shown in Figure \ref{fig:M3CUnderbraced}. This shows that it is necessary for the stress to vary continuously with respect to the configuration. 
\end{remark}
\begin{figure}[t]
\centering
\includegraphics[scale=.4]{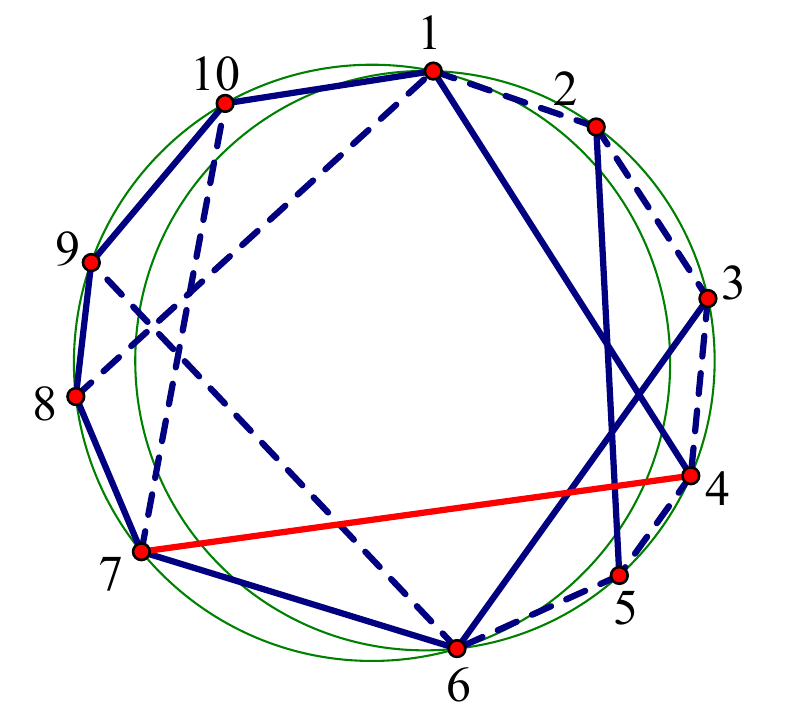}
\captionsetup{labelsep=colon,margin=1.3cm}
\caption{The underlying  braced polygon graph of this framework is minimally 3-connected. However, $e_{47}$ has zero stress when vertices $\{1,2,3,4,5,6\}$ are placed on one circle and vertices $\{6,7,8,9,10,1\}$ are placed on another circle. The set of convex stressed configurations is disconnected for this braced polygon graph so we cannot vary the stress continuously from a properly stressed configuration to this configuration. 
}
\label{fig:M3CUnderbraced}
\end{figure}

In order to interpolate between our result and the work of Geleji and Jord{\'a}n, 
 we will give a combinatorial proof for the equivalence between the unique interval property and minimal 3-connectivity for generic rigidity circuits in the appendix.

\begin{figure}[t]
\centering
\includegraphics[scale=1]{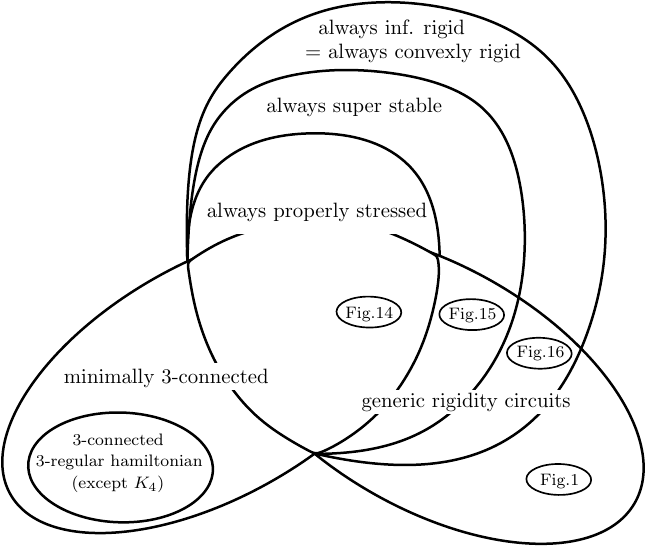}
\caption{Relation among properties of braced polygon graphs. 
}
\label{fig:diagram}
\end{figure}

The reader may wonder if all braced polygon graphs which satisfy one or more of the properties listed in  Theorem~\ref{thm:minimally} are   minimally 3-connected. This is not the case.
There are generic rigidity circuits which satisfy (a), (b), (c), or (d) in Theorem~\ref{thm:minimally} that are not minimally connected.
Figure~\ref{fig:diagram} summarizes the results.

\clearpage
\section{Examples and Questions}\label{Sect: Examples}
In the figures of braced polygonal graphs throughout this section, we will adopt the convention that all strictly convex plane realisations of the graphs shown have a stress which is positive on the dashed blue edges and negative on the solid blue edges. These stresses take positive values on red edges at some configurations and negative values on red edges at other configurations.

 \subsection{Superposition technique}
Consider the following basic question for convex braced polygons.
\begin{question}
    How does one tell if a braced polygon graph is universally rigid in all convex configurations? 
\end{question}
    
    One idea is that two braced polygon graphs known to be super stable in all convex configurations can be ``stuck" together through superposition to create another super stable  braced polygon graph. When two such graphs share 3 vertices not on a line, an edge on the bounding polygon of one (that necessarily has a positive stress) can cancel with a brace in the other (that necessarily has a negative stress). This is shown in Figure \ref{fig:superposition} where a convex braced heptagon graph is assembled by adding three smaller braced polygon graphs while deleting two edges. The fact that the resulting braced polygon graph is super stable in all convex configurations follows from the following result of Connelly, which can be proved using the same techniques as his proof of \cite[Lemma 4]{polygon}.

    \begin{lem}\label{lem:super}
    Let $(G,p)$ be a framework in general position in $\mathbb{R}^d$ and $G_1,G_2$ be subgraphs of $G$ such that $G=G_1\cup G_2$ and $|V(G_1)\cap V(G_2)|\geq d+1$. Suppose that  $G_i$ is super stable for some equilibrium stress $\omega_i$ for both $i=1,2$. Let $\omega:E(G)\to \mathbb{R}$ by putting $\omega(e)=\omega_1(e)+\omega_2(e)$, and taking $\omega_i(e)=0$ when $e\not\in E(G_i)$. Then $\omega$ is an equilibrium stress for $G$ and  $G$ is super stable for  $\omega$.       
    \end{lem}
    
    \begin{figure}[t]
    \centering
    \includegraphics[scale=0.5]{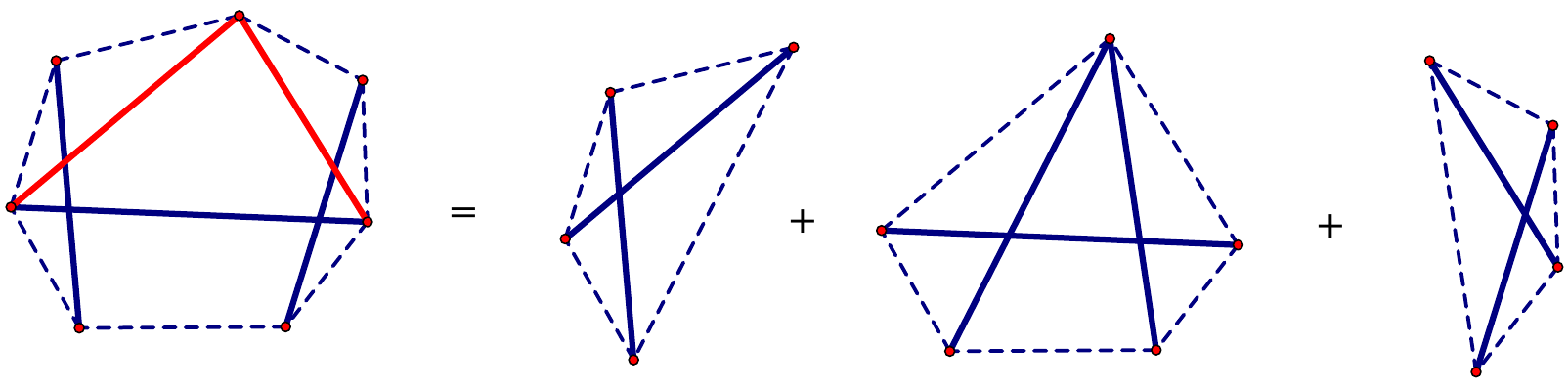}
    \captionsetup{labelsep=colon,margin=1.3cm}
    \caption{An example of assembling a braced polygon graph which is super stable in all strictly convex configurations from three smaller ones. 
    } \label{fig:superposition}
    \end{figure}

    The technique is not restricted to convex polygons, and sometimes it is possible to have super stable frameworks as in Figure \ref{fig:dificient}, where the vertices are in a strictly convex position, but not all the boundary edges of the corresponding convex polygon are part of the framework. Can you see how the superposition works here? In such a convex position, the boundary polygon edges and braces always have the corresponding positive and negative stresses, when in a strictly convex position.
\begin{figure}[t]
\centering
\includegraphics[scale=0.4]{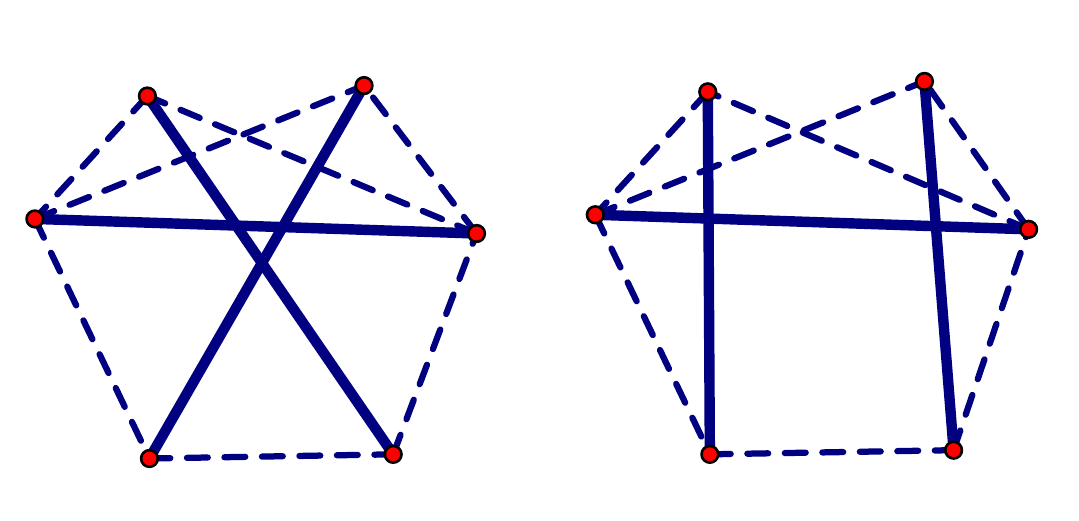}
\captionsetup{labelsep=colon,margin=1.3cm}
\caption{These frameworks are always super stable, via superposition, when their vertices are in a strictly convex configuration.  The top part of Figure \ref{fig:Flipped} shows the derivation for one of these frameworks.} \label{fig:dificient}
\end{figure}

\subsection{Braced polygonal circuits with 7 vertices}
As a concrete example and application of our observations, we examine all braced polygonal circuits
with $7$ vertices.
Figures~\ref{fig:proper},~\ref{fig:super}, and~\ref{fig:convexly} show the examples of all braced polygonal circuits
with $7$ vertices, collected into the categories of being always super stable with a proper stress (Figure \ref{fig:proper}), always super stable, but with some configurations that have positive stresses for internal braces (Figure \ref{fig:super}), and those that are not globally rigid in the plane, but are convexly rigid in the plane (Figure \ref{fig:convexly}).
The lists are constructed based on Theorem~\ref{thm:minimally}, Theorem~\ref{m3p}, and the superposition technique.

\begin{figure}[t]
\centering
\includegraphics[scale=0.45]{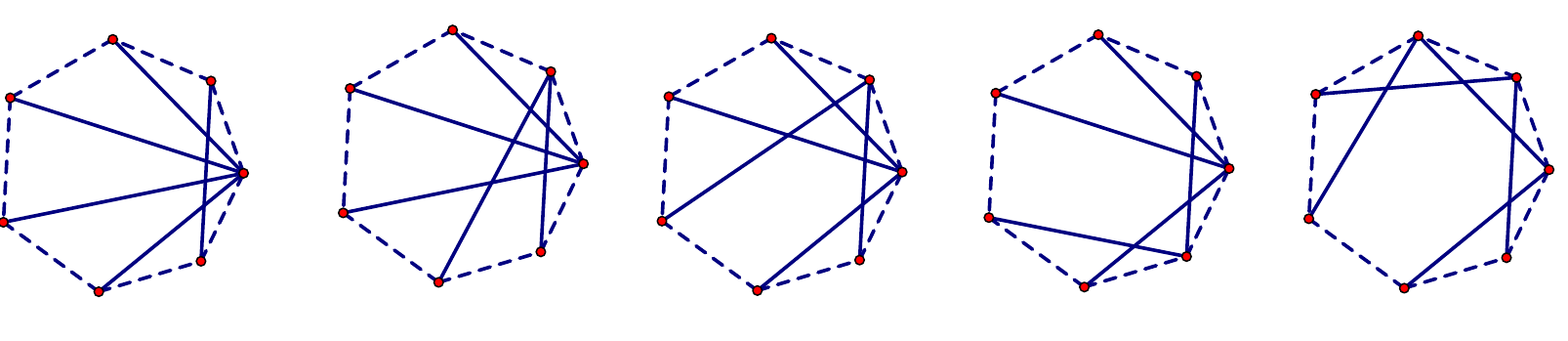}
\captionsetup{labelsep=colon,margin=1.3cm}
\caption{These are all the braced heptagonal circuits with a proper stress
in all strictly convex configurations.  See Theorem \ref{m3p}.} 
\label{fig:proper}
\end{figure}

\begin{figure}[p]
\centering
\includegraphics[scale=0.7]{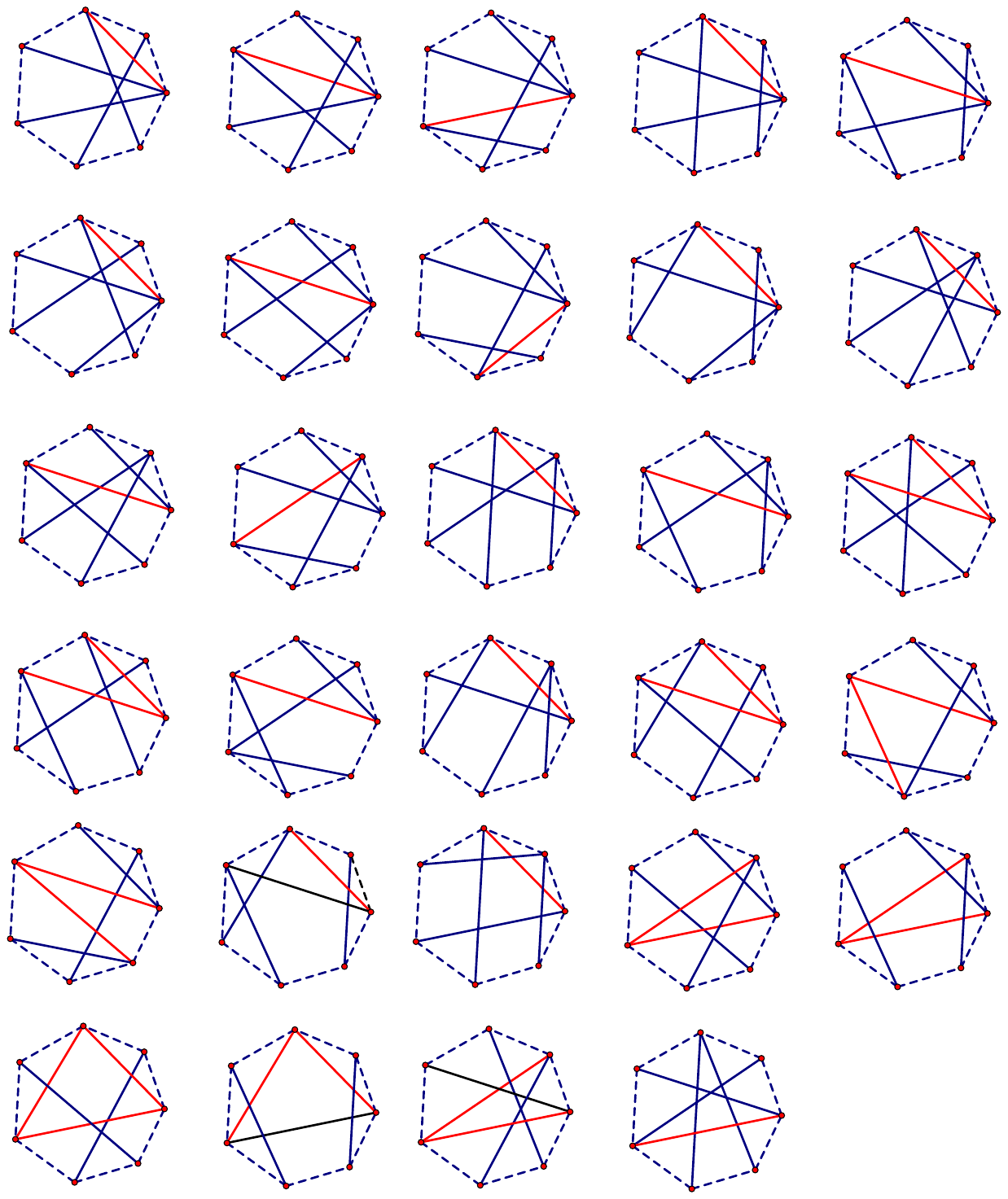}
\captionsetup{labelsep=colon,margin=1.3cm}
\caption{These are all the braced heptagonal circuits which are  super stable in all strictly convex configurations, but where some of the braces of the configurations may have a positive stress. 
} 
\label{fig:super}
\end{figure}

\begin{figure}[t]
\centering\includegraphics[scale=0.6]{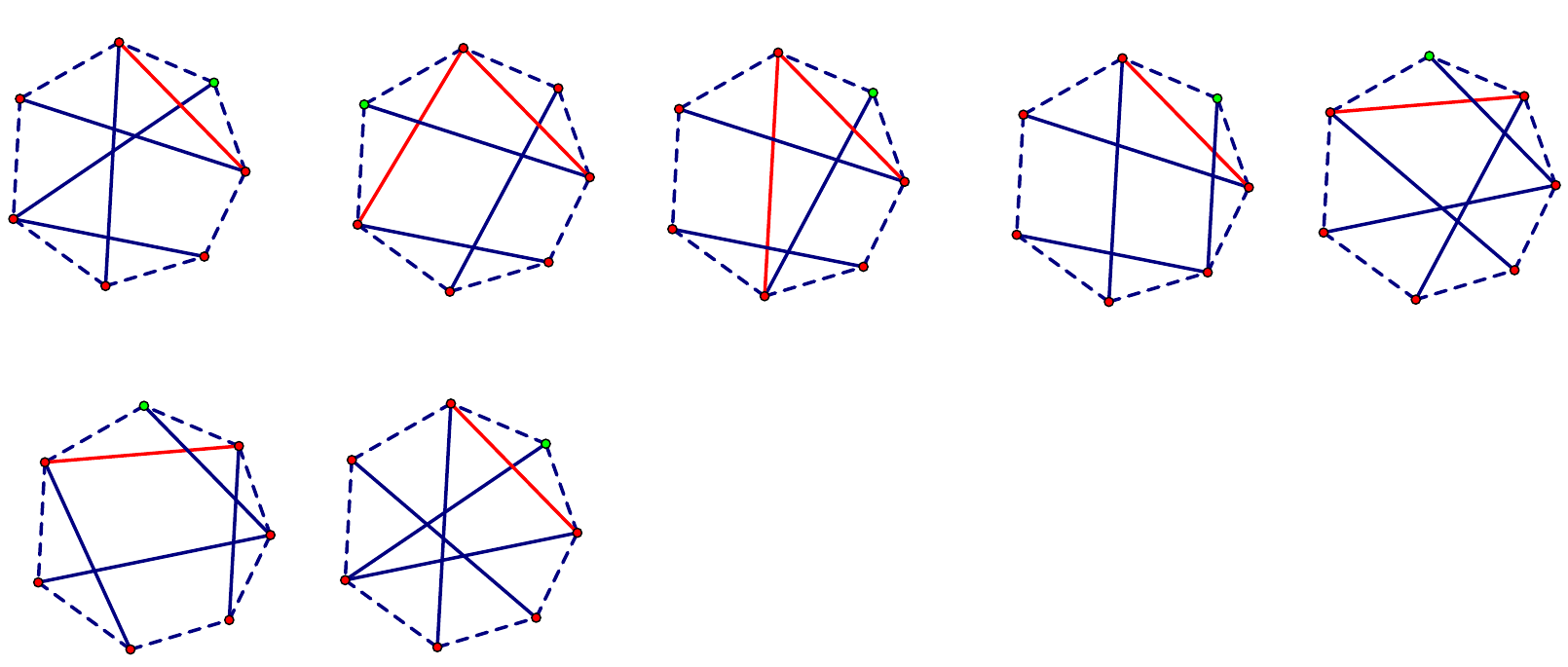}
\captionsetup{labelsep=colon,margin=1.3cm}
\caption{These are all the braced heptagonal circuits, which are convexly rigid  but are not globally rigid at all strictly convex plane configurations. The blue vertices may have realizations in the plane, where they are inside the convex hull of other vertices and, in that case, the configuration is not globally rigid in the plane.  The idea is to move the green vertices to the interior of the corresponding red edge, and then there is a $2$-dimensional space of stresses, indicating a non-trivial infinitesimal motion, such that one of the implied equivalent but not congruent configurations is convex and the other is not. See Figure \ref{fig:non-global}, which is related to the second framework here. } 
\label{fig:convexly}
\end{figure}

The red edges in Figure \ref{fig:super} are defined as those internal braces that do not  a negative stress in all strictly convex configurations.  It is necessary, in that case, that both end-vertices of the edge have degree at least four but this is not always sufficient. Indeed, in the last two cases in the second column, there is an edge that has both end-vertices of degree four which is not colored red.  In order for a stress on an edge to remain negative for all convex configurations, 
from Theorem~\ref{stressExist},
it is necessary and sufficient, that, when the edge is removed, the graph should fail to be $3$-connected.

\subsection{Braced polygonal circuits with 8 vertices}
In Figure~\ref{fig:proper-8}, we give a list of braced polygonal circuits with 8 vertices satisfying the condition of Theorem~\ref{m3p}.
For all strictly convex braced polygons with seven or fewer vertices, we can decide whether there is an example that is not globally rigid in the plane as a bar framework, or in all cases it is super stable.  However, we have found some strictly convex cases with $8$ vertices,  where we know that the framework is not super stable, and by Theorem $12.1$ in \cite{iterative} not even universally rigid in some strictly convex cases. But we cannot decide whether they are globally rigid or not for the non-universally rigid strictly convex cases in the plane. 
A notable example is the underlying graph of Figure \ref{fig:Grunbaum}, which is the orthogonal projection of a 3-dimensional framework given by Gr{\"u}nbaum~\cite{Grunbaum-book} onto the plane.
 It appears that all the strictly convex plane configurations of this graph  are  globally rigid, but we have no proof. 

\begin{figure}[t]
\centering
\includegraphics[scale=0.6]{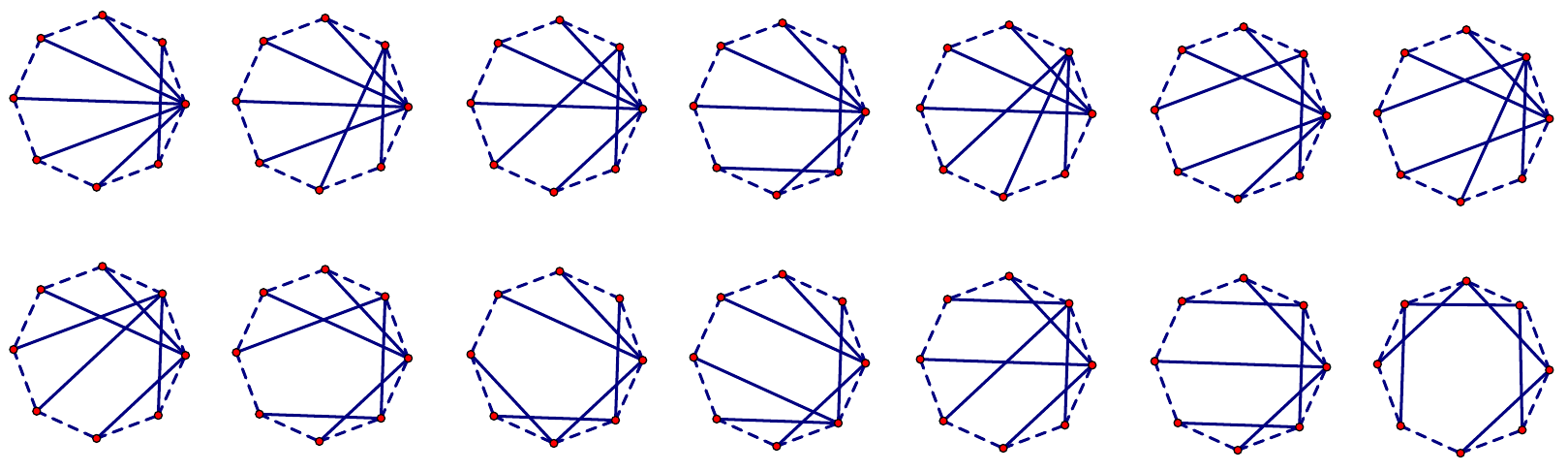}
\captionsetup{labelsep=colon,margin=1.3cm}
\caption{These are all the braced polygonal circuits with 8 vertices having  a proper stress in all strictly convex configurations.  See Theorem \ref{m3p}.} 
\label{fig:proper-8}
\end{figure}

In this graph,
the vertex sets $\{1,4,5,8\}$ and $\{2,3,6,7\}$ each induces an infinitesimally rigid framework in any strictly convex configuration in the plane, since they are constructed of two triangles with a common edge.  If there is another configuration of these (infinitesimally) rigid subsets of the configuration, then the other non edges $15$ and $37$ can only decrease in length.  If either one of these lengths does not decrease in length, then the bottom row of Figure \ref{fig:Flipped} shows that with this added bar, the graph would be universally rigid at all strictly convex configurations.  Nevertheless, with this added information, we still don't see how to decide whether all the strictly convex configurations of the original graph are globally rigid.


\begin{figure}[t]
    \centering
    \includegraphics[scale=0.5]{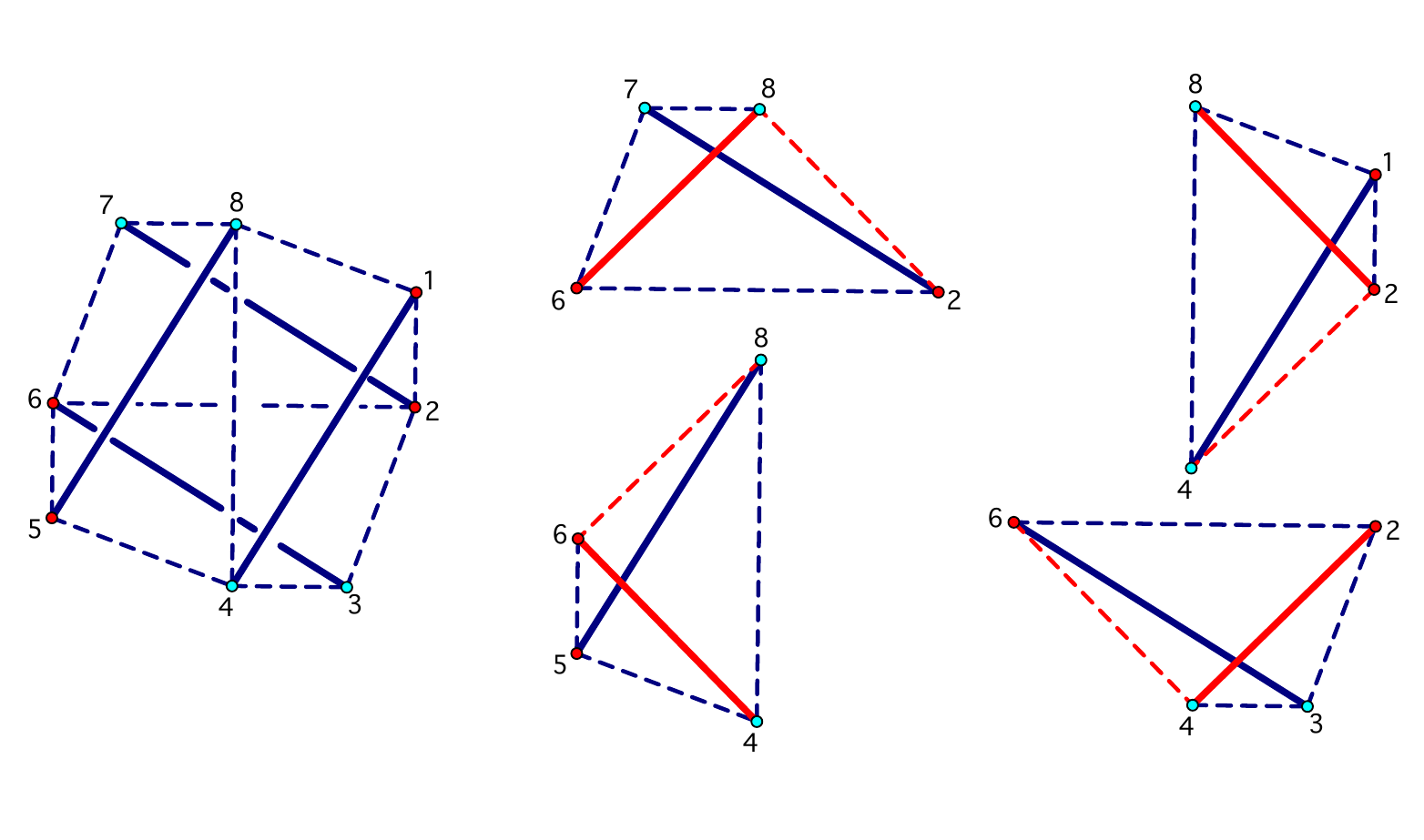}
    \captionsetup{labelsep=colon,margin=1.3cm}
    \caption{
    The left framework is the orthogonal projection of a $3$-dimensional  framework 
    from Gr{\"u}nbaum and Shephard's ``Lectures in Lost Mathematics"~\cite{Grunbaum-book} onto the plane. 
    It has a 4-fold rotational symmetry and edge 78 is parallel to edge 26. 
    It is a superposition of the four planar quadrilaterals shown on the right. The original 3-dimensional framework is known to be super stable, see  \cite[Section 11.7]{book}, but its projection onto the plane is only universally rigid. We do not know whether all the strictly convex plane realisations of its underlying graph are globally rigid.}\label{fig:Grunbaum}
\end{figure}

\begin{figure}[t]
\centering
\includegraphics[scale=0.5]{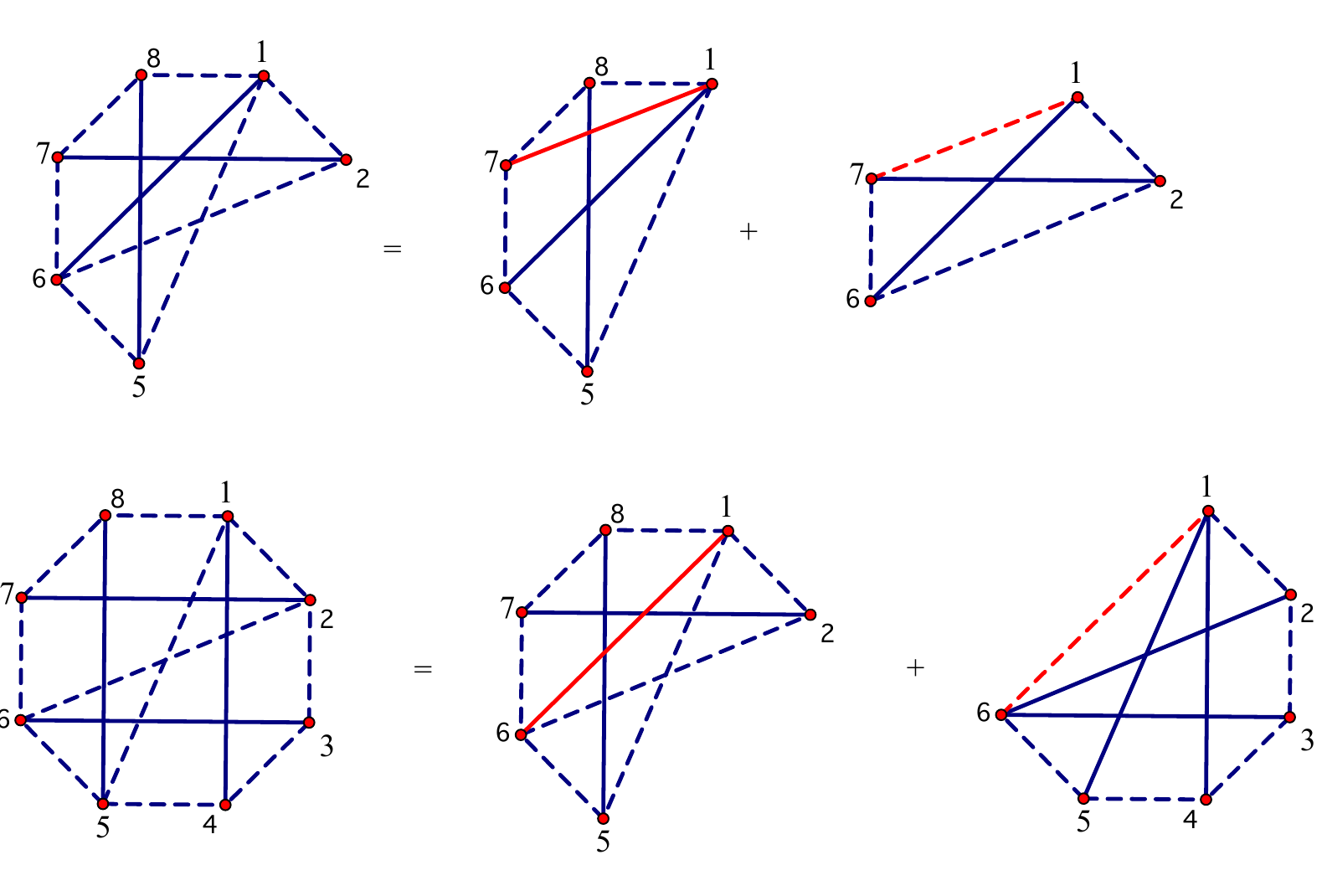}
\captionsetup{labelsep=colon,margin=1.3cm}
\caption{These graphs are always super stable, via superposition, when their vertices are in a strictly convex configuration.   Here the red edges are the ones that always cancel. Notice that the bottom graph  is the same as the graph of the Gr\"unbaum framework of Figure  \ref{fig:Grunbaum}, except that the $84$ and $51$ edges have been switched in the plane projection. } \label{fig:Flipped}
\end{figure}

   Figure \ref{fig:visual} shows that there is a region of the space of convex configurations of the graph of the Gr\"unbaum framework that are not universally rigid, and other regions that are super stable, assuming a rotational symmetry of 90 degrees. The red region corresponds to the positon of vertex $7$, where the configuration is not universally rigid, but the polygon is still strictly convex. It turns out that when the point $7$ is below the $2,6$ line in Figure \ref{fig:visual}, the configuration is not universally rigid.

\begin{figure}[p]
\centering
\includegraphics[scale=0.7]{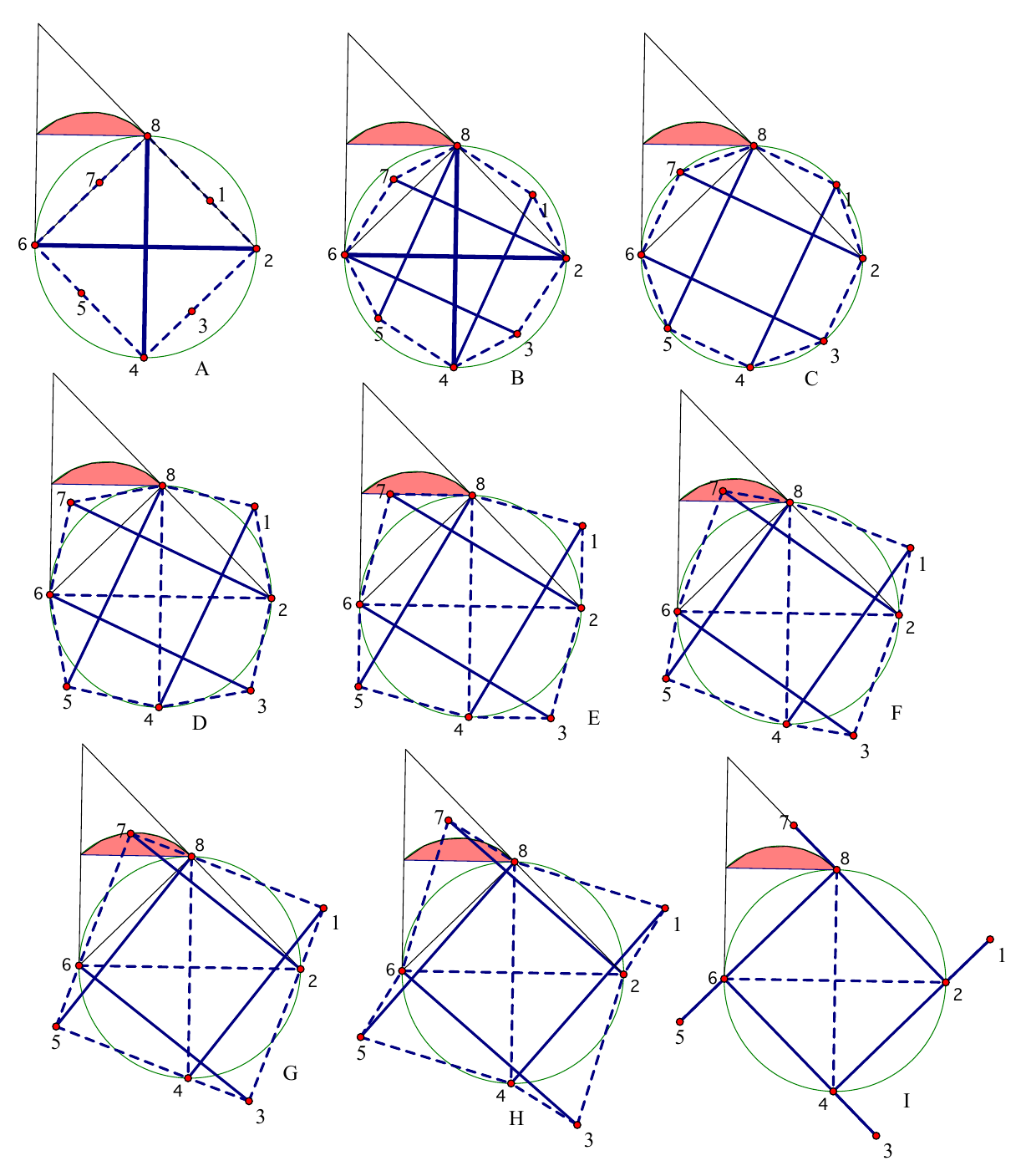}
\captionsetup{labelsep=colon,margin=1.3cm}
\caption{These frameworks are various plane configurations of the 
braced polygon graph given in Figure \ref{fig:Grunbaum}. 
The edges that have $0$ stress are omitted.
 The $0$ stressed bars still contribute to the bar framework's infinitesimal rigidity, and they are all infinitesimally rigid as bar frameworks if we restore the omitted edges. 
 See Table~\ref{table} for more detailed analysis of each case.
 } 
\label{fig:visual}
\end{figure}

\begin{table}[t]
\begin{tabular}{ |l|l| } 
 \hline
 A & This essentially is a square with two braces as diagonals, and so it is super stable.  \\ 
 \hline
 B & This continues to be super stable, but $\omega_{62}<0$  
 \\ 
 \hline
 C & Point $7$ and all the other points lie on the same circle and so there is a \\
  & reflection symmetry in the whole configuration that has proper stress with \\
  & $\omega_{84}=\omega_{62}=0$ and all other stresses on the internal edges negative.\\
 \hline
 D & These configurations must be super stable until Point $7$ crosses the horizontal line,\\
  & and $\omega_{84}=\omega_{62}>0$ by looking ahead to configurations E.  \\
 \hline
 E & This is the Gr\"unbaum framework in Figure \ref{fig:Grunbaum}, \\
  & so the stress matrix $\Omega$ is PSD with only one extra $0$ eigenvalue. \\
 \hline
 F & The circular arc of the red region is part of a circle with center at the midpoint\\
  &  between the points $6$ and $8$. So the angle at Point $7$ is at most a right angle and the\\
  &   configuration is convex. Looking ahead to the configuration I, the $0$ eigenvalue became negative. \\
 \hline
 G & None of the eigenvalues or stress signs change, but the region is not strictly convex \\
 \hline
 H & Now the region is not convex, and none of the eigenvalues or stress signs have changed \\
 \hline
 I & $\omega_{67}=0$ and so the value of the stress energy is the negative of the stress\\ 
  & energy of a square with diagonals as braces.\\
 \hline
\end{tabular}
\caption{Case Analysis from Figure \ref{fig:visual}.}
\label{table}
\end{table}

Another simple way to create examples of universally rigid frameworks is to adjoin a sequence of universally rigid ``plates" together with a brace from the first to last as in Figure \ref{fig:plates}.
One of the problems with applying this idea more generally is that each of the plates needs to be universally rigid, at least as a bar framework.  If there is an open set of configurations that is universally rigid, except for a triangle in the plane, if more than one of the plates is not a triangle, it will have a non-zero stress, and then the union will have a $2$-dimensional set of stresses.  We are treating the case of circuits, here, which have only a one-dimensional space of stresses.  The last case here in Figure \ref{fig:plates} is the next to last example in the first column, in the table of Figure \ref{fig:super}.

\begin{figure}[t]
\centering
\includegraphics[scale=0.5]{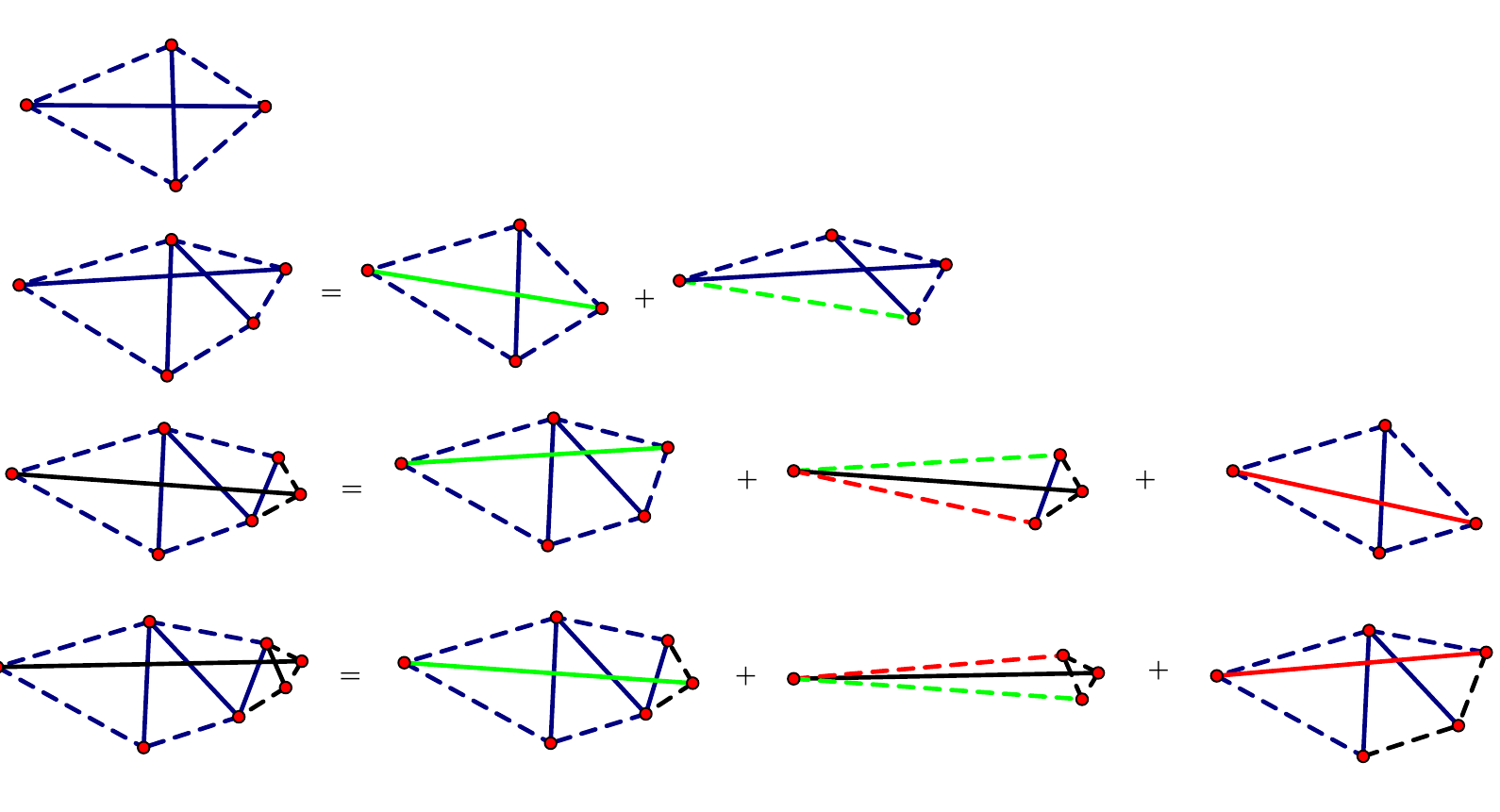}
\captionsetup{labelsep=colon,margin=1.3cm}
\caption{The frameworks on the left
are made of rigid triangles placed edge-to-edge with a brace connecting the first to last going through the interior of all the triangles, and thus are easily seen to be universally rigid, just by the triangle inequality. This shows how to construct such a framework as a superposition of simpler frameworks.  The red and green members, when superimposed, have their stresses cancel.} 
\label{fig:plates}
\end{figure}

\begin{question}
    Is there a strictly convex configuration of the braced polygon graph in Figure \ref{fig:Grunbaum} that is not globally rigid? 
\end{question}

    Note that this graph is infinitesimally rigid at all strictly convex configurations and thus the stress matrix (and its eigenvalues) vary continuously as the configuration varies continuously in the plane.
    Using this fact, we can look for  non-universally rigid configurations in the plane close to that given in Figure~\ref{fig:Grunbaum}. 
      The configuration in Figure~\ref{fig:Grunbaum} is universally rigid (this can be proved using a careful analysis of the superposition technique), but not super stable (because 
      there is a 3-dimensional framework having the same one-dimensional space of stresses).
      However, there are small perturbations of the configuration in the plane, where the stress matrix picks up a negative eigenvalue, and as a result, we loose universal rigidity.  A random search among strictly convex configurations for this graph in the plane is unlikely to find these non-universally rigid configurations, even though they form an open subset of such configurations.
    

\subsection{Further questions for braced polygonal circuits}
The central question, possibly wishful thinking, is whether there is a way to determine global rigidity for braced polygonal circuits, by the methods presented here. We have a way to tell when such a circuit is NOT globally rigid by Theorem \ref{convex}, such as the example in Figure \ref{fig:non-global}. Minimally 3-connected cases can be solved using Theorem \ref{m3p}, such as those in Figure \ref{fig:proper}. When neither of the theorems apply, our best hope is to use the superposition technique in Figure \ref{fig:superposition}. This technique works for most of the cases that we have tried. There are still open cases such as Figure \ref{fig:Grunbaum}. This leads to the following:

\begin{question}\label{boundary-question} If a braced polygonal circuit $G$ has a  strictly convex plane realisation $(G,\p)$ which is not globally rigid, is it true that there is a convex, but possibly not strictly convex, realisation $(G,\q)$ with a non-trivial infinitesimal flex $\q'$ such that $\p = \q+\q'$?  
\end{question}                 
We are asking if every non~globally rigid, strictly convex plane realisation of a braced polygonal circuit can be constructed as in Figures \ref{fig:Jackson} and \ref{fig:non-global}.
If this is true, then the braced octagonal graph  of Gr\"unbaum, Figure \ref{fig:Grunbaum} 
would  be globally rigid at all strictly convex configurations. 

\begin{question}\label{numbers}
We saw in Figures \ref{fig:super} and \ref{fig:proper} that the number of braced polygonal circuits that are  super stable in all strictly convex configurations is significantly larger that the number which have a proper stress in all strictly convex configurations.   Does this persist with larger numbers of vertices? What is the asymptotic behavior as the number of vertices increases? 
\end{question}
If the number of vertices $n$ is large, it is likely that the situation in Figure \ref{fig:Jackson} will be the dominant case. There are more ways for subgraphs to be stressed in a convex shape with large $n$. If two subgraphs can be stressed at the same time, then the configuration will have two stresses. As a result, the configuration will not even be convexly rigid. If we only look at convexly rigid  braced polygonal circuits with $n$ vertices, then is it possible to estimate how many of them are globally rigid, universally rigid, or properly stressed at all strictly convex configurations?

\subsection{Further questions for general convex braced polygons}
For a specific configuration, global rigidity is known to be NP-hard, while infinitesimal rigidity can be determined by computing the null space of a matrix. According to Theorem \ref{convex}, these two are identical in the set of strictly convex configurations for braced polygon graphs. 
 A question is whether this fact can be used to design an efficient algorithm.
\begin{question} \label{inf-rigid} For a braced polygon graph $G$ that is generically rigid in the plane, is there a reasonable polynomial-time algorithm to determine when all strictly convex braced polygon of $G$ are infinitesimally rigid in the plane?
\end{question} 
For example, one way to construct an example of a convex braced polygon that is infinitesimally rigid in all the strictly convex configurations is to build from a triangle inductively by connecting a new vertex of degree $2$ to the previous vertices, a Type I Henneberg operation. However, this procedure does not generate all such examples, as demonstrated in Figure \ref{fig:convex-rigid}.

\begin{figure}[t]
\centering
\includegraphics[scale=0.3]{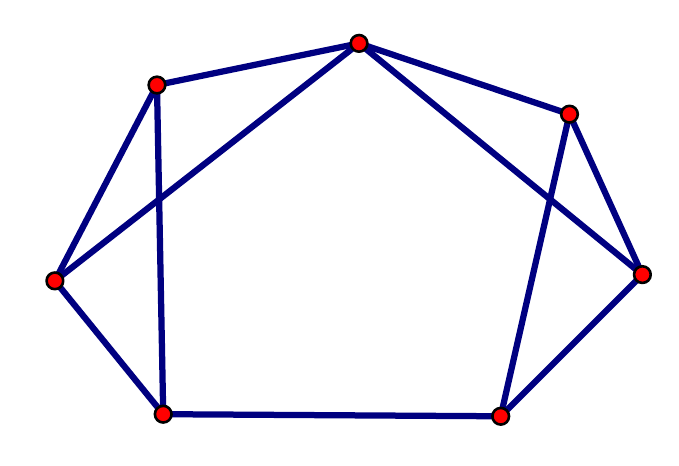}
\captionsetup{labelsep=colon,margin=1.3cm}
\caption{This shows a strictly convex braced polygon that is infinitesimally rigid in all strictly convex configurations, but it is not obtained just from Henneberg Type I operations. } \label{fig:convex-rigid}
\end{figure}

\begin{question} Is there an analogue of Theorem \ref{convex} for $\mathbb{R}^3$ (or in higher dimensions)? Namely if all the vertices of two bar frameworks in $\mathbb{R}^3$, each of whose vertices lie on their convex hull, when is there a rotation of one of them such that their average is convex. 
\end{question}

When the framework is the one skeleton of a strictly convex triangulated sphere in $\mathbb{R}^3$ it is known \cite{book} that framework is convexly rigid by Cauchy's theorem and infinitesimally rigid by Dehn's theorem. On the other hand, there are examples of non-convex triangulated surfaces, whose vertices lie on their convex hull (weakly convex) that are always infinitesimally rigid. See \cite{Schlenker,Schlenker-Connelly}. Are they convexly rigid?


\begin{question}\label{flexible}
Finally, rigidity cannot be done without asking the most obvious question:
    is there a braced polygon that is generically rigid but flexible? 
\end{question}
    Without convexity, it is possible to place all edges of Figure \ref{fig:Grunbaum} on a rectangular grid for flexibility. With convexity, it is not obvious. Our intuition says this should not be possible just like in Cauchy's Rigidity Theorem. Theorem \ref{convex} shows that strictly convex configurations of generically rigid braced polygon graphs are closely related to convex polytopes. Grasegger and Legersky's idea on NAC coloring\cite{Grasegger2017GraphsWF} can assert the existence of flexible configurations, yet convexity rules out the affine motions they use to construct these configurations.

\newpage
\bibliographystyle{abbrvnat}
\bibliography{braced.bib}

\appendix

\section{Appendix: unique interval property}
In this appendix, we shall discuss a relation between minimal 3-connectivity and the unique interval property introduced by Geleji and Jord{\'a}n~\cite{robust}.

 Without loss of generality, we assume that the vertices are labeled counterclockwise on the boundary of the outside Hamiltonian circuit for braced polygons. 
\begin{defn}\label{interval}
    We say that a braced polygon $G$ satisfies the 
    \textit{unique interval property} \\\cite{robust} if \vspace{-\topsep}\begin{enumerate}
        \item there exists a connected segment of the outside polygon $I$ of vertices $\{v_1,...,v_k\}$ where $k\ge 1$ such that $\deg (v_i)\ge 4$ for $v_i\in I$ and $(v_{i-1},v_{i+1})$ is an edge, 
        \item $\deg(v_i)=3$ for $v_i\in V-I$ and no other edges exist between two vertices in $V-I$ except $(v_n,v_2)$ when $k=1$, and
        \item if two edges connecting $I$ to $V-I$ intersect in the interior of the polygon, then the two vertices in $I$ are connected by an edge of the outer polygon. 
    \end{enumerate}
\end{defn}

\begin{thm}\label{thm:unit_interval}
Let $G$ be a braced polygonal circuit.
Then $G$ is minimally 3-connected if and only if 
$G$ satisfies the unique interval property.
\end{thm}

For the proof, we introduce the idea of a dual polygon for minimally 3-connected braced polygons. 
\begin{defn}
For a brace $e$ in a minimally 3-connected braced polygon graph $G=(V,E)$, a \textit{dual brace} $e'$ is a pair of vertices that disconnect $(V,E-\{e\})$. A \textit{dual polygon graph} $G^*$ of $G$ is the polygon with each brace $e$ of $G$ replaced by a dual brace $e'$.
\end{defn}

Observe that in order for the dual brace to disconnect $(V,E-\{e\})$, each dual brace crosses exactly one brace. Therefore, no two braces can share the same dual brace. Since there are equal numbers of braces and dual braces, each brace crosses exactly one dual brace. Note also that for any given brace, there may be more than one way to choose a dual brace. An example of a convex braced polygon and its dual is given in figure \ref{fig:dualPoly}.

\begin{figure}[t]
\centering
\includegraphics[scale=0.4]{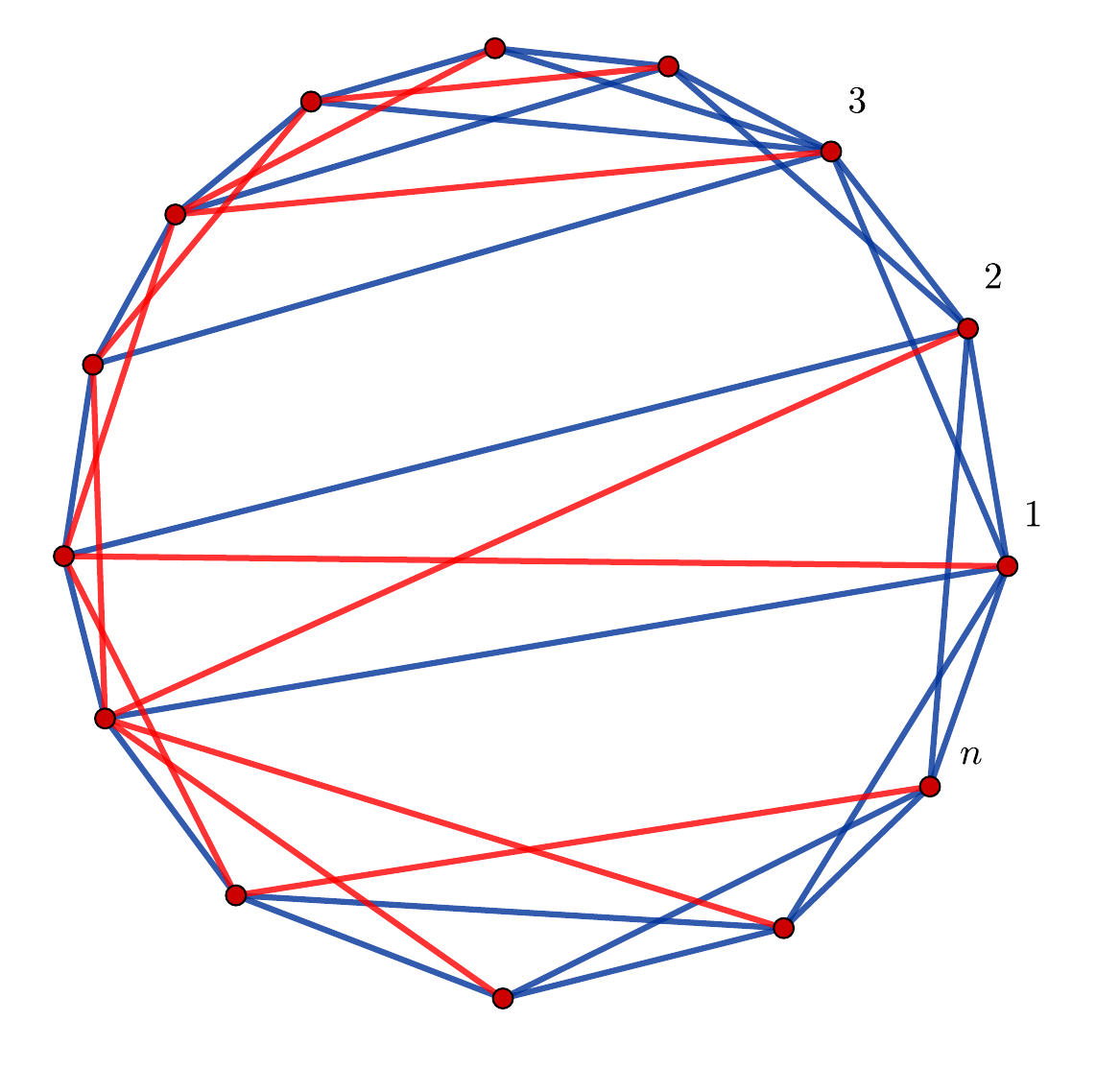}
\captionsetup{labelsep=colon,margin=1.3cm}
\caption{A convex braced polygon in blue with the dual braces in red} \label{fig:dualPoly}
\end{figure}

To prove Theorem \ref{m3p}, we first need some properties of a dual polygon. 

\begin{proposition}\label{dualsegmentcount}
Let $G$ be a minimally 3-connected braced polygonal circuit with $n$ vertices. For every connected segment of $k<n$ ($k>1$) vertices in the outer polygon, there can be at most $k-2$ dual braces in the induced subgraph on this segment of a dual polygon $G^*$. 
\end{proposition}
\begin{proof}
We proceed by induction. The case for $k=2$ is trivial. Suppose that the proposition is true up to $k-1$ vertices, we prove it for $k$. 

Consider the segment $\{1,2,...,k\}$. Since $G$ is 3 vertex connected, there exists a brace $(\alpha,\beta)$ such that $\beta\in\{2,...,k-1\}$ and $\alpha\in\{k+1,...,n\}$. Considering the two subsegments $\{1,...,\beta\}$ and $\{\beta,...,k\}$, they can have at most $\beta-2$ and $k-\beta-1$ dual braces in the corresponding subsegments according to the induction hypothesis. At most one dual brace can cross $(\alpha,\beta)$. Therefore, we can have at most 
\begin{equation*}
    (\beta-2)+(k-\beta-1)+1=k-2
\end{equation*}
dual braces on the segment $\{1,...,k\}$. 
\end{proof}

\begin{proposition}\label{dual3vc}
Let $G$ be a minimally 3-connected braced polygonal circuit with $n$ vertices. Each dual polygon $G^*$ of $G$ is 3 vertex connected.
\end{proposition}
\begin{proof}
We proceed by contradiction. Suppose that the proposition is false and $\{1,k\}$ disconnects $G^*$. In other words, no dual brace can cross $\{1,k\}$ and all dual braces are distributed within two subsegments. 

First, observe that $(1,k)$ cannot be a dual brace itself. Otherwise, the segment $\{1,2,...,k\}$ has at most $k-2$ dual braces, and the segment $\{k,k+1,...,1\}$ has at most $n-k$ dual braces by proposition \ref{dualsegmentcount}. Since $(1,k)$ exists in both segments, there are at most $n-3$
dual braces in $G^*$, contradicting $G$ being a circuit. 

Next, each induced subgraph of $G^*$ on $\{1,2,...,k\}$ and $\{k,k+1,...,1\}$ must be 2 vertex connected. Otherwise, some vertex further splits a segment into 2 subsegments. By proposition \ref{dualsegmentcount}, that segment cannot have maximum number of dual braces. However, we need both segments to have the maximum number of dual braces, i.e. $k-2$ and $n-k$, so that the dual graph must have $n-2$
dual braces as proposed. 

Finally, since $G$ is 3 vertex connected, there exists a brace $e$ from $\{2,...,k-1\}$ to $\{k+1,k+2,...,n\}$. However, since both $\{1,2,...,k\}$ and $\{k,k+1,...,1\}$ are 2 vertex connected, $e$ must cross two dual braces. This gives us a contradiction. 
\end{proof}

\begin{proposition}\label{d3c}
Let $G$ be a minimally 3-connected braced polygonal circuit with $n$ vertices. All degree 3 vertices of $G$ are consecutive on the outer polygon. 
\end{proposition}
\begin{proof}
Suppose that we have a segment of $G$, $\{k+1,...,n\}$, with all vertices having degree 3. If $k=1$, then the degree 3 vertices are trivially consecutive. We consider $k>1$.  If every vertex of the segment is connected to $\{1,...,k\}$ through a brace, then there are at most $n-k$ braces with one end in the segment. Therefore, the induced subgraph of $G$ on $\{1,...,k\}$ has the maximal number: $k-2=(n-2)-(n-k)$ braces. For each brace that has both ends in $\{k+1,...,n\}$, the number of braces in $\{1,...,k\}$ increases by one. Therefore, there cannot be any internal brace in the segment $\{k+1,...,n\}$ in order for $\{1,...,k\}$ to have at most $k-2$ braces. 

Suppose that there is a segment of $G$, $\{k+1,...,n\}$, of maximal length with every vertex being degree 3, we can assume vertices $1,k$ have degree at least 4 and that the segment $\{1,...,k\}$ has $k-2$ braces in $G$.  We prove the following: if the induced subgraph of $G$ on a segment $\{1,...,k\}$ has $k-2$ braces, and if both $1$ and $k$ has degree at least 4, then each vertex $2,...,k-1$ must have degree at least 4. 

We proceed by induction on $k$. The claim is trivial for $k=2$. Suppose that the claim is true for segments with length $1,...,k-1$, we show it for length $k$. 

Since $G^*$ is 3-vertex connected by proposition \ref{dual3vc}, there exists a dual brace $(\alpha,\beta)$ such that $\alpha\in\{k+1,...,n\}$ and $\beta\in\{2,...,k-1\}$. Since only one brace can cross $(\alpha,\beta)$, each subsegment $\{1,...,\beta\}$ and $\{\beta,...,k\}$ must have maximum number of braces, $\beta-2$ and $k-\beta-1$ respectively. If $\beta$ has degree at least $4$, then the proof is finished by the induction hypothesis. 

If $\beta$ has degree $3$, then $(\beta-1,\beta+1)$ can be a dual brace. This is demonstrated in Figure \ref{fig:prop24}. There are two cases: $\beta\in\{3,...,k-2\}$ or $\beta\in\{2,k-1\}$. In either case, we show a contradiction. 

\begin{figure}[t]
\centering
\includegraphics[scale=0.4]{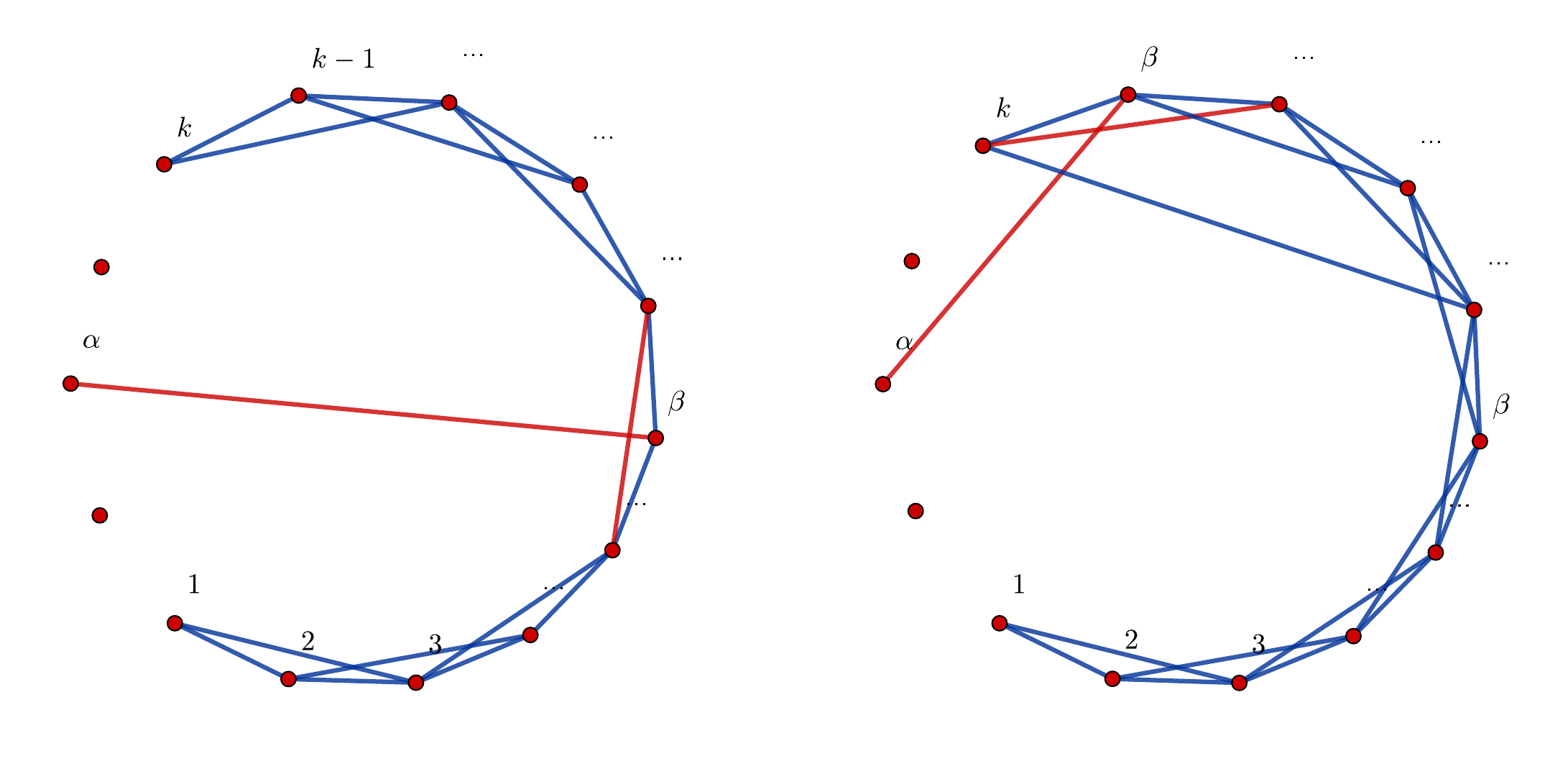}
\captionsetup{labelsep=colon,margin=1.3cm}
\caption{The braces are blue and the dual braces are red. If $\beta$ is in the middle of the segment $(1,...,k)$, the segment needs 3 more braces but it can only have 2 more crossing the two dual braces. If $\beta=k-1$, then the brace connected to $k$ must stay inside $\{1,...,k\}$. In this case, $(\alpha,k)$ will separate the graph} \label{fig:prop24}
\end{figure}

Case 1: If $\beta\in\{3,...,k-2\}$, then $\{1,...,\beta-1\}$, $\{\beta-1,...,\beta\}$, $\{\beta,...,\beta+1\}$, and $\{\beta+1,...,k\}$ are four subsegments separated from each other by the two dual braces $(\alpha,\beta)$ and $(\beta-1,\beta+1)$. They can have at most $\beta-3$, $0$, $0$, $k-\beta-2$ braces, respectively, within each subsegment. Furthermore, $2$ more braces can have two ends in distinct subsegments by crossing a dual brace. This gives only $(\beta-3)+(k-\beta-2)+2=k-3$ braces on the segment $\{1,...,k\}$, contradicting our assumption that $\{1,...,k\}$ has $k-2$ braces. This case is on the left side of Figure \ref{fig:prop24}.

Case 2: If $\beta=2$ or $\beta=k-1$, without loss of generality, we can assume $\beta=k-1$. By our assumption of $k$ having degree at least 4, $\alpha\not=k+1$. We proceed to show the contradiction that $G$ cannot be $3$ vertex connected. The segment $\{1,...,k-1\}$ has exactly $k-3$ braces because only one brace can cross $(\alpha,\beta)$. In order for $\{1,...,k\}$ to have $k-2$ braces, the brace whose dual is $(\alpha,\beta)$ must be in the segment $\{1,...,k\}$. If $k\ge 4$, then the brace with one end $k-1$ has its other end in $\{1,...,k-2\}$ in order for the number of braces in $\{1,...,k-1\}$ to be $k-3$. In this case, $(\alpha,k)$ separates the graph as demonstrated by the right side of Figure \ref{fig:prop24}. If $k=3$, $(1,3)$ is the only possible brace with dual $(\alpha,\beta)$. Furthermore, $\alpha\not=4,n$ when $k=3$ because vertex $1$ has degree at least 4 by our assumption. As a result, either $(\alpha, 3)$ or $(\alpha,1)$ will separate the graph, depending on whether the brace with one end $2$ has the other end in $\{4,5,...,\alpha\}$ or $\{\alpha+1,\alpha+2,...,n\}$.
\end{proof}

\begin{proposition}\label{d4c}
Let $G$ be a minimally 3 vertex connected braced polygonal circuit with $n$ vertices. If the segment $\{n,1,2,...,k,k+1\}$ has $k$ braces and vertices $1,...,k$ all have degree at least 4 in $G$, then $(n,2),(1,3),...,(k-1,k+1)$ are braces of $G$. 
\end{proposition}
\begin{proof}
We proceed by induction on $k$. The claim is trivial for $k=1$. Suppose that the proposition is true for $1,...,k-1$, we prove it for $k$. 

Since the $G^*$ is 3-vertex connected by proposition \ref{dual3vc}, there exists a dual brace $(\alpha,\beta)$ such that $\alpha\in\{k+2,...,n-1\}$ and $\beta\in\{1,...,k\}$. Because only one brace can cross $(\alpha,\beta)$, each smaller segment $\{n,1,...,\beta\}$ and $\{\beta,...,k,k+1\}$ must have maximum number of braces in $G$.

If $\beta\in\{2,...,k-1\}$, by induction hypothesis, $(n,2),...,(\beta-2,\beta)$ and $(\beta,\beta+2),...,(k-1,k+1)$ are braces. It remains to show $(\beta-1,\beta+1)$ is a brace. We have used $k-1$ braces, so there must be only one remaining brace on the segment $\{n,1,...,k,k+1\}$. Since both subsegments $\{n,1,...,\beta\}$ and $\{\beta,...,k,k+1\}$ have the maximum number of braces, the remaining one must cross the dual brace $(\alpha,\beta)$. 

\begin{figure}[t]
\centering
\includegraphics[scale=0.4]{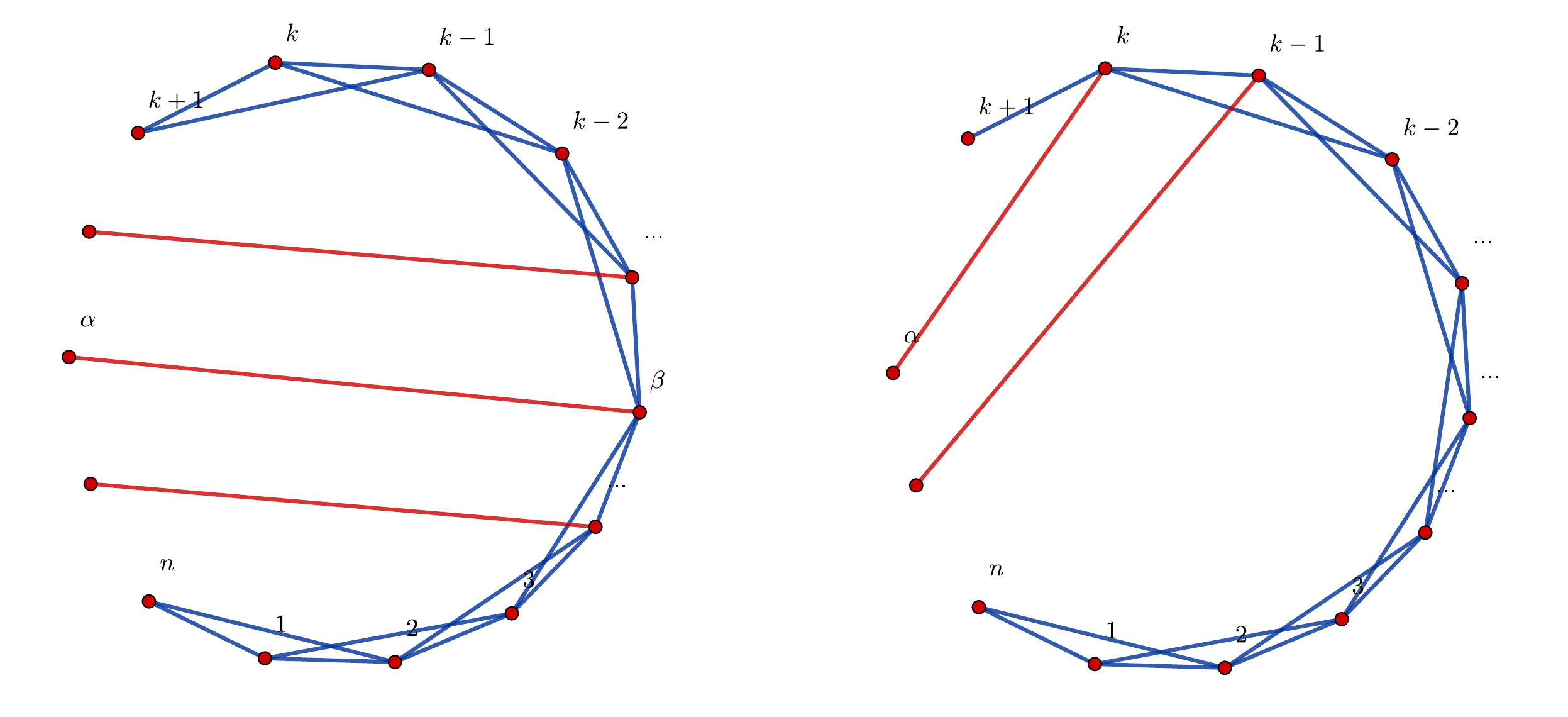}
\captionsetup{labelsep=colon,margin=1.3cm}
\caption{ The dual braces cannot have their left ends in either $k+1$ or $n$. Otherwise, either $1$ or $k$ cannot have degree at least 4} \label{fig:prop25}
\end{figure}

Consider the dual braces of $(\beta,\beta-2)$ and $(\beta,\beta+2)$, they must have one end in $\{\beta-1,\beta+1\}$ and the other end in $\{k+1,...,n\}$. If both of them have the other end being either $k+1$ or $n$, then $(k+1,n)$ is a valid candidate for the last brace in $\{1,...,k\}$, otherwise $(\beta-1,\beta+1)$ is the only choice. The earlier case cannot occur because both $1$ and $k$ have degree at least $4$. This is shown on the left of Figure \ref{fig:prop25}. 

If $\beta=1$ or $\beta=k$, without loss of generality we assume $\beta=k$, by the induction hypothesis, $(n,2),...,(k-2,k)$ are braces. It remains to show $(k-1,k+1)$ is a brace. By counting we can have only one more brace and one end of it must be $k+1$. Consider the dual brace of $(k-2,k)$. One end of it must be $k-1$, and the other end can only be one of $k+2,...,n-1$ so that vertex $1$ and $k$ can have degree at least $4$. This is shown on the right of Figure \ref{fig:prop25}. In this case, $(k-1,k+1)$ must be a brace. 
\end{proof}

\medskip

We are now ready to prove Theorem~\ref{thm:unit_interval}.

\medskip

\noindent\textbf{Proof of Theorem~\ref{thm:unit_interval}.}
Suppose that $G$ is minimally 3-connected.
Notice that Propositions \ref{d3c} and \ref{d4c} give the first two parts of the unique interval property. To show the last property, suppose that the two lines $(v_i,v_k)$ and $(v_j,v_l)$ intersect in the interior, $v_i,v_j\in I$ and $v_k,v_l\in V-I$. Without loss of generality, we assume $i<j$. If $|j-i|\ge 2$, then deleting edge $(i,i+2)$ (guaranteed to exist by Propositions \ref{d3c} and \ref{d4c}) will not affect 3-connectivity. To see this, deleting the edge $(i,i+2)$ only exposes $v_{i+1}$ to the interior, but no dual brace from $v_{i+1}$ can separate the graph by connecting $v_{i+1}$ to $V-I$ without crossing $(v_i,v_k)$ or $(v_j,v_l)$. 
Hence, $G$ satisfies the unique interval property.

Conversely, suppose that $G$ satisfies the unique interval property with respect to an interval $I$.
Since each vertex in $V-I$ has degree three,
$G-e$ is not 3-connected if an edge $e$ is incident to a vertex in $V-I$.
Suppose that an edge $e=(i,i+2)$ is deleted with $i, i+2\in I$. Then the last condition of the universal interval property guarantees that a line through $i+1$ can separate the graph.
Hence $G$ is minimally 3-connected.
\qed

\end{document}